\documentclass[12pt]{article}
\usepackage{amsthm,amsmath,amsfonts,amssymb}
\usepackage{verbatim}
\usepackage[alphabetic]{amsrefs}
\usepackage{graphicx}

\newtheorem{thm}{Theorem}
\newtheorem{lemma}[thm]{Lemma}
\newtheorem{cor}[thm]{Corollary}
\newtheorem{prop}[thm]{Proposition}
\newtheorem{definition}[thm]{Definition}
\newtheorem{remark}[thm]{Remark}

\newcommand{\R}{\mathbb{R}}
\newcommand{\Rn}{\mathbb{R}^n}

\newcommand{\abs}[1]{\ensuremath{\vert #1 \vert}}

\newcommand{\Lpi}{\mathcal{L}_\infty^+}
\newcommand{\Lmi}{\mathcal{L}_\infty^-}
\newcommand{\Li}{\mathcal{L}_\infty}
\renewcommand{\ll}{\Lambda_\infty^\alpha}
\newcommand{\dd}{\partial}
\newcommand{\e}{\varepsilon}

\DeclareMathOperator*{\hr}{\Gamma}

\DeclareMathOperator*{\spt}{supp}
\DeclareMathOperator*{\dist}{dist}
\DeclareMathOperator*{\dive}{div}

\begin{document}
\title{FRACTIONAL EIGENVALUES}
\author{Erik Lindgren \qquad Peter Lindqvist}
\date{Department of Mathematical Sciences\\Norwegian University of Science and Technology\\NO-7491
  Trondheim, Norway}
\maketitle
\begin{abstract}\noindent 
 \textsf{We study the non-local eigenvalue problem
    $$
    2\, \int_{\Rn}\frac{|u(y)-u(x)|^{p-2}\bigl(u(y)-u(x)\bigr)
  } {|y-x|^{\alpha
      p}}\,dy +\lambda |u(x)|^{p-2}u(x)=0
    $$ 
    for large values of $p$ and derive the limit equation as $p\to\infty$. Its viscosity solutions have many interesting properties and the eigenvalues exhibit a strange behaviour.} 
  \end{abstract}
\section{Introduction}

 The problem of
minimizing the fractional Rayleigh quotient 
\begin{equation}\label{rayq}
{\displaystyle \underset{\phi}{\inf}}\,\frac{\displaystyle \int_{\Rn}\!\!\int_{\Rn}\frac{|\phi(y)-\phi(x)|^p}{|y-x|^{\alpha p}}\,dx\,dy}{\displaystyle
\int_{\Rn}|\phi(x)|^p\,dx}
\end{equation} among all functions $\phi$
in the class $C_0^{\infty}(\Omega), \, \phi \not \equiv 0$ leads to an
interesting eigenvalue problem with the non-local Euler-Lagrange 
equation
$$
2\, \int_{\Rn}\frac{|u(y)-u(x)|^{p-2}\bigl(u(y)-u(x)\bigr)
  } {|y-x|^{\alpha
      p}}\,dy +\lambda |u(x)|^{p-2}u(x)=0
$$
in a bounded domain $\Omega$ in the $n$-dimensional
Euclidean space.
Here $ p \geq 2$ and $ n < \alpha p <n+p$. It is an essential feature
that the solutions may be multiplied by constant factors. We treat the
solutions in the viscosity sense and prove, among other things, that
positive viscosity solutions are unique (up to a
normalization) and that the first eigenvalue is isolated. For sign changing solutions we detect some strange
phenomena, caused by the influence of points far away appearing in the
domain of integration for the non-local operator. Indeed, it is as if
the nodal domains were interacting with each other. In the linear case
$p = 2$ the connexion to the more familiar \emph{fractional Laplacian}
is the principal value formula
$$ (-\Delta)^{(2 \alpha-n)/2}u\,(x) = -
C(n,\alpha)\,\textup{P.V.}\int_{\Rn}\frac{u(y)-u(x)}{|y-x|^{2\alpha}}\,dy$$
valid at least in the range $n<2\alpha<n+2$. The linear case  has
 been treated in \cite{Kwa12}, \cite{FL11} and \cite{ZRK07}. 

To the best of our knowledge, no advanced regularity theory is yet
available for $p \not = 2$. To assure continuity for eigenfunctions we
have, occasionally, assumed that $\alpha p$ is larger than what
appears to be necessary. This is of little importance here, because
our main interest is the asymptotic case $p= \infty$. Formally, one
has then to minimize the quotient
\begin{equation*}
\displaystyle\frac{\displaystyle 
  \left\|\frac{u(y)-u(x)}{|y-x|^{\alpha}}\right \|_{L^{\infty}(\Rn
    \times \Rn)}}{\displaystyle\|u\|_{L^{\infty}(\Rn)}} , \qquad 0<\alpha\leq 1,
\end{equation*}
among all admissible functions $u$. However, this minimization problem
has too many solutions. Therefore the proper limit equation is called
for. The equation takes the form
\begin{equation}
\label{limeqq}
{\displaystyle
  \max\left\{\mathcal{L}_{\infty}u\,(x),\,\mathcal{L}_{\infty}^{-}u\,(x)
    + \lambda u(x)\right\}\,=\,0}
\end{equation}
in $\Omega$. In this new equation $\lambda$ is a real parameter (the eigenvalue) and
\begin{align*}
\mathcal{L}_{\infty}u\,(x)&=
\underset{y\in\Rn}{\sup}\,\frac{u(y)-u(x)}{|y-x|^{\alpha}}  \,
\,+\,\,\underset{y\in\Rn}{\inf}\,\frac{u(y)-u(x)}{|y-x|^{\alpha}}\\
\mathcal{L}_{\infty}^{-}u\,(x)&= \underset{y\in\Rn}{\inf}\,\frac{u(y)-u(x)}{|y-x|^{\alpha}}.
\end{align*}
The solutions $u$, referred to as $\infty$-eigenfunctions, belonging to $C_0(\overline{\Omega})$, and regarded
as zero outside $\Omega$, have to be interpreted in the viscosity
sense, because the operator $\mathcal{L}_{\infty}u\,(x)$ is not
sufficiently smooth. It is remarkable that the parameter $\lambda$
behaves like a genuine eigenvalue. Indeed, a non-negative solution 
 exists if and only if $\lambda$
has the value:
$$\lambda \, = \, \frac{1}{\bigl(\underset{x \in
    \Omega}{\max}\dist(x,\R^n\setminus\Omega)\bigr)^{\alpha}}\quad\equiv\quad
\Lambda_{\infty}^{\alpha}.$$
Thus the radius $R$ of the largest inscribed ball in $\Omega$ is
decisive: $\Lambda_{\infty}^{\alpha} = R^{-\alpha}$.  
If $\alpha = 1$, the eigenvalue  $\Lambda_{\infty}^{1}
= \Lambda_{\infty}$ is, incidentally, the same as the one in the
differential equation
\begin{equation}
\label{globq}
\max \left\{ \Lambda_{\infty}- \frac{|\nabla u|}{u},\,\,
    \sum_{i,j}\frac{\partial u}{\partial x_{i}} \frac{\partial u}{\partial x_{j}}
 \frac{\partial^{2}  u}{\partial x_{i} \partial x_j}\right\}\,\, =
\,\, 0, 
\end{equation}
treated in \cite{JLM99}, but the equations are not equivalent. This
differential equation is related to finding 
$$\underset{u}{\min}\,\frac{\|\nabla
  u\|_{L^{\infty}(\Omega)}}{\|u\|_{L^{\infty}(\Omega)}}$$
among all $u \in W^{1,\infty}_0(\Omega),\, u\not \equiv 0$. It is the
limit of the Euler-Lagrange equations coming from the minimization of
the Rayleigh quotients 
\begin{equation}
\label{peter}
\displaystyle \frac{\displaystyle\int_{\Omega}|\nabla u(x)|^{p}\,dx}{\displaystyle\int_{\Omega}|
  u(x)|^{p}\,dx}
\end{equation}
as $p \to  \infty$. Therefore a comparison of the two problems is of
actual interest.

Let us return to the $\infty$-eigenvalue equation (\ref{limeqq}). A central part
of the domain $\Omega$, called the High Ridge, is important. With the notation
$\delta(x) = \dist(x,\Rn\setminus \Omega)$ and $R = \|\delta\|_{\infty}$,
 the set $\Gamma = \{x\in \Omega|\, \delta(x) = R\}$ is the High
 Ridge. We have discovered the remarkable representation formula
$$u(x) = \frac{\delta(x)^{\alpha}}{\delta(x)^{\alpha}+
  \rho(x)^{\alpha}},$$
where $\rho(x) = \dist(x,\Gamma)$. The formula is valid in every
domain and gives a first $\infty$-eigenfunction. If $\Gamma_{1} 
\subset \Gamma$ is an arbitrary non-empty closed subset, the same 
formula, but with $\rho(x)$ replaced by
$$\rho_{1}(x) =  \dist(x,\Gamma_{1}),$$
also yields an $\infty$-eigenfunction. Thus uniqueness is lost.  We do not know
whether all positive solutions of (\ref{limeqq}) are represented. ---No such
formula is known for the differential equation (\ref{globq}). To derive and verify the representation
formula we use the Dirichlet problem for the equation
$$\mathcal{L}_{\infty}u\,(x) = 0 \quad \text{in} \quad \Omega
\setminus \Gamma$$
with boundary values $0$ and $1$. This equation has been treated
in \cite{CLM11}.

We have included a brief account on the higher eigenvalues,
corresponding to sign changing solutions. In this case the
$\infty$-eigenvalue equation (\ref{limeqq})
has to be amended to include the open set $\{u<0\}$ and the nodal line
$\{u=0\}$, see equation \eqref{limeqhigh} on page \pageref{limeqhigh}. Strange phenomena occur. First, the nodal domains,
which are the connected components of the open sets $\{u > 0\}$ and $\{u
< 0\}$,
do not have the same \emph{first} $\infty$-eigenvalue, yet they all come
from the same higher $\infty$-eigenfunction. Second, the restriction
of a higher $\infty$-eigenfunction to one of its nodal domains (and
extended as zero) is not an $\infty$-eigenfunction for the nodal
domain in question. Even one-dimensional examples exhibit this, see
Section \ref{sec:1d}.

To this one may add that such a behaviour is totally
impossible for equations like 
$$\Delta u + \lambda u = 0,\quad
\dive(|\nabla u|^{p-2}\nabla u) + \lambda |u|^{p-2}u = 0,$$  and (\ref{globq}). It
is the non-local character of our equation that causes such phenomena.

Needless to say, there are many open problems with our fractional,
non-local, non-linear eigenvalue problem, both for finite exponents
$p$  and for $p= \infty$. For example, the simplicity of the first
$\infty$-eigenvalue $\Lambda_{\infty}^{\alpha}$ is valid only in the
special case when the High Ridge contains exactly one
point. Nonetheless, this does not yet exclude the possibility that the 
minimizers of the fractional Rayleigh quotient (\ref{peter}) can
converge to a unique function, as $p \to \infty$. It stands to reason
that the limit procedure $p \to \infty$ should produce the maximal
solution, the one with $\Gamma_{1} = \Gamma$.  But  the presently known situation for the ``local'' problem
\eqref{globq} is also incomplete; see however \cite{Yu07} and \cite{CPJ09} for some progress. The higher
eigenvalues are mysterious when $p \not = 2$: for none of the
equations mentioned is it known that the eigenvalues are countable!
This challenging problem about the spectrum is likely to be the most difficult open question in this connection.

\emph{Acknowledgement:} We thank Evgenia Malinnikova
for helping us to verify an inequality.

\section{Preliminaries and Notation}

To study the fractional Rayleigh  quotient \eqref{rayq} the so-called
fractional Sobolev spaces\footnote{These spaces are also known as
  Aronszajn, Gagliardo or Slobodeckij spaces} $W^{s,p}(\R^n)$ with $0<s<1$ are
expedient. If $1<p<\infty$, as usual, the norm is defined through
$$
\|u\|^p_{W^{s,p}(\R^n)}=\int_{\R^n}\!\!\int_{\R^n} \frac{|u(y)-u(x)|^{p}}{|y-x|^{ sp+n}} d x d y+\int_{\R^n}|u|^p d x.
$$
The space $W^{s,p}(D)$ for a bounded and open
 subset $D$ of $\R^n$ is defined similarly and, as usual $W^{s,p}_0(D)$ is defined as the closure of $C_0^\infty(D)$ with respect to the norm $\|\cdot \|_{W^{s,p}(D)}$.
The relation between $s$ and our $\alpha$ is $n+sp=\alpha p$.
In ``Hitchhiker's Guide to the Fractional Sobolev Spaces'' one can
find most of the useful properties, cf.  \cite{NPV11}. We list some of them below.
 
\begin{thm}[Sobolev-type inequality] Let $D\subset \R^n$ be bounded and open, $sp<n$ and $s\in (0,1)$. Then there is a constant $C$ such that
$$
\|u\|_{L^{p^*}(\R^n)}\leq C \left(\int_{\R^n}\!\!\int_{\R^n} \frac{|u(y)-u(x)|^{p}}{|y-x|^{ sp+n}} d x d y\right)^\frac1p, 
$$
for all $u\in W_0^{s,p}(D)$ and where $p^{*}=\frac{np}{n-sp}$.
\end{thm}
 This is Theorem 6.10 on page 49 in \cite{NPV11}. From this one can
 extract the following estimate.

\begin{thm}\label{omegaest} Let $\alpha p > n$. If $\Omega$ is a bounded domain in
  $\Rn$ there exists a constant $C(n,p,\alpha) >0 $ such that $$
C(n,p,\alpha)\, |\Omega|^{1- \frac{\alpha p}{n}} \int_{\Omega}|\phi|^p
\,dx \leq \int_{\Omega}\!\int_{\Omega}\frac{|\phi(y)-\phi(x)|^p}{|y-x|^{\alpha
      p}}\,dx\,dy $$
for all $\phi \in C_{0}^{\infty}(\Omega)$.
\end{thm}

The right-hand side is the so-called Gagliardo seminorm raised to the
$p^{\text{th}}$  power.

\begin{thm}[H\"older embedding]\label{ca} Let $D\subset \R^n$ be
  bounded and open, $sp>n$ and $s\in (0,1)$. Then there is a constant
  $C$ such that for all $u\in W_0^{s,p}(D)$
$$
\|u\|_{C^{0,\beta}(\R^n)}\leq C\|u\|_{W^{s,p}(\R^n)}, 
$$
 where $\beta=(sp-n)/p$.
\end{thm}
This is Theorem 8.2 on page 38 in \cite{NPV11} and here
$$
\|u\|_{C^{0,\alpha}(D)}=|[u]|_{\alpha,D}+\|u\|_{L^\infty(D)},
$$
where we use the notation
$$
|[u]|_{\alpha,D}=\left\|\frac{u(x)-u(y)}{|x-y|^\alpha}\right\|_{L^\infty(D\times D)},\qquad |[u]|_\alpha=|[u]|_{\alpha,\R^n}.
$$

\begin{thm}[Compact embedding]\label{cpt} Assume $D\subset \R^n$ to be bounded and open, $p\in [1,\infty)$ and $s\in (0,1)$. Let $u_i$ be a sequence of functions in $W_0^{s,p}(D)$ such that 
$$
 \|u_i\|_{W^{s,p}(\R^n)}\leq M<\infty.
$$
Then there is a subsequence of $u_i$ converging in $L^q(\R^n)$ for all $q\in [1,p]$.
\end{thm}
This result can be found in Theorem 7.1 on page 33 in \cite{NPV11}.

It is worth mentioning that asymptotically, as $s\to 1$, the space $W^{s,p}$ becomes $W^{1,p}$, see \cite{BBM02}. The same also holds for the corresponding Euler-Lagrange equation, see \cite{IN10}.

 A function $u\in C_0(\overline{\Omega})$ or $u\in W_0^{s,p}(\Omega)$ is always assumed be defined in the whole $\R^n$ by extending it by zero.

\section{The Euler-Lagrange Equation}

Let $\Omega$ be a bounded domain in $\Rn$. We consider the problem of
minimizing the fractional Rayleigh quotient among all functions $\phi$
in the class $C_0^{\infty}(\Omega), \, \phi \not \equiv 0:$
\begin{equation}\label{ray}
{\displaystyle \underset{\phi}{\inf}}\,\frac{\displaystyle \int_{\Rn}\!\!\int_{\Rn}\frac{|\phi(y)-\phi(x)|^p}{|y-x|^{\alpha p}}\,dx\,dy}{\displaystyle
\int_{\Rn}|\phi(x)|^p\,dx}\quad =\quad \lambda_1.
\end{equation}
It is desirable that $$n < \alpha p < n+p,$$ but we will often require
the narrower bound  $$n < \alpha p < n+p -1.$$ Occasionally, we take
$\alpha p > 2n$ (instead of $>n$) to guarantee regularity. We 
aim at studying
the asymptotic case $p \to \infty.$ For $p$ large enough, any exponent
$0 < \alpha \leq 1$ is sooner or later included. The usual fractional
Sobolev space $W^{s,p}$ has the exponent $n+sp$ in the place of our
$\alpha$, i.e. $$s = \alpha - \frac{n}{p}, \quad 0 < s < 1.$$ For us $\alpha$
is more convenient. It is helpful to keep in mind that in the range
$\alpha p > n$ one has
$$\underset{|y-x|>1}{\int\!\!\int}\frac{dx\,dy}{|y-x|^{\alpha p}} <
\infty, \qquad
\underset{|y-x|<1}{\int\!\!\int}\frac{dx\,dy}{|y-x|^{\alpha p}} =
\infty.$$
The inequality
\begin{equation}
\label{sobomes}
C(n,p,\alpha)\, |\Omega|^{1- \frac{\alpha p}{n}} \int_{\Omega}|\phi|^p
\,dx \leq \int_{\Rn}\!\!\int_{\Rn}\frac{|\phi(y)-\phi(x)|^p}{|y-x|^{\alpha
      p}}\,dx\,dy
\end{equation}
shows that the infimum $\lambda_1 > 0.$ We call  $\lambda_1$ \emph{the
  first eigenvalue}\footnote{The name ``principal frequency'' is
  synonymous.}. It is worth noting that, although $\phi = 0$ in the
whole complement $\Rn \setminus \Omega,$ the identity
\begin{gather*}
\int_{\Rn}\!\!\int_{\Rn}\frac{|\phi(y)-\phi(x)|^p}{|y-x|^{\alpha
      p}}\,dx\,dy\\ = \int_{\Omega}\!\int_{\Omega}\frac{|\phi(y)-\phi(x)|^p}{|y-x|^{\alpha
      p}}\,dx\,dy + 2\,\int_{\Rn \setminus \Omega}dy \!\!\int_{\Omega}\frac{|\phi(x)|^p}{|y-x|^{\alpha
      p}}\,dx
\end{gather*}
has a term from the complement. However, the inequality
(\ref{sobomes}) is valid also with $\Omega \times \Omega$ as the
domain of integration in the double integral, see Theorem 2. But the minimization problem is not quite the same
if $\Rn \times \Rn$ is replaced by $\Omega \times \Omega$ in the
integral. Our choice has the advantage that
the property
$$\lambda_1(\Omega) \leq \lambda_1(\Upsilon), \quad \text{if}\quad \Upsilon
  \subset \Omega$$
is evident for subdomains. A simple change of coordinates yields that
$$\lambda_1(\Omega) = k^{\alpha p -n} \lambda_1( k \Omega), \qquad k >
0.$$
This and (\ref{sobomes}) indicate that \emph{small domains have large first
eigenvalues.}

A minimizer of the fractional Rayleigh quotient \eqref{ray} cannot change sign, since 
$$
|\phi(y)-\phi(x)|>\big||\phi(y)|-|\phi(x)|\big|\quad \text{when}\quad
 \phi(y)\phi(x)<0.$$ The minimizer in the next theorem is called \emph{the first
  eigenfunction} in $\Omega.$

\begin{thm} There exists a non-negative minimizer $u \in
  W_{0}^{s,p}(\Omega), \, u \not \equiv 0,$ and $u = 0$ in $\Rn
  \setminus \Omega.$ It satisfies the Euler-Lagrange equation
\begin{equation}
\label{euler}
\int_{\Rn}\!\!\int_{\Rn}\frac{|u(y)-u(x)|^{p-2}\bigl(u(y)-u(x)\bigr)
  \bigl(\phi(y)-\phi(x)\bigr)} {|y-x|^{\alpha
      p}}\,dx\,dy = \lambda \int_{\Rn}|u|^{p-2}u\phi\,dx
\end{equation}
with $\lambda = \lambda_1$ whenever $\phi \in C_0^{\infty}(\Omega).$
If $\alpha p > 2 n,$ the minimizer is in $C^{0,\beta}(\Rn)$ with
$\beta = \alpha - 2n/p.$
\end{thm} 

\begin{proof} The existence of a minimizer is proved via the direct
method in the Calculus of Variations. First a minimizing sequence of
admissible functions $\phi_j$ is selected. It can be normalized so
that
$\|\phi_j\|_{L^p(\Rn)} = \|\phi_j\|_{L^p(\Omega)} = 1.$ Then we have
$$ \int_{\Rn}\!\!\int_{\Rn}\frac{|\phi_j(y)-\phi_j(x)|^p}{|y-x|^{\alpha
      p}}\,dx\,dy + \int_{\Rn}\!|\phi_j|^p\,dx \leq \lambda_1+1+1$$
for large indices $j.$ 
According to Theorem \ref{cpt}, there is a subsequence that converges in $L^p(\Rn)$.  
The limit of the subsequence, say $u$, is in
$W^{s,p}_{0}(\Omega)$ and vanishes outside $\Omega$. Fatou's lemma yields that  $u$ is minimizing. So is \emph{a fortiori} $|u|.$ Thus the existence of a non-negative
minimizer is proved.

To derive the Euler-Lagrange equation,  one uses a device due
to Lagrange. If $u$ is minimizing, consider the competing function 
$$v(x,t) = u(x) + t \phi(x), \quad \phi\in C_0^\infty(\Omega).$$ The necessary condition
\begin{equation*}
\frac{\textrm{d}}{\textrm{d}\,t}\left\{\frac{\displaystyle\int_{\Rn}\!\!\int_{\Rn}\frac{|v(y,t)-v(x,t)|^p}{|y-x|^{\alpha
      p}}\,dx\,dy}{\displaystyle\int_{\Rn}|v(x,t)|^p\,dx}\right\} = 0 \quad \text{at}
\quad t = 0
\end{equation*}
for a minimum yields the equation (\ref{euler}).

Finally, the $\beta-$H\"{o}lder continuity is a property of the fractional
Sobolev space, cf. Theorem \ref{ca}. This concludes our proof. \end{proof}

\medskip
The Euler-Lagrange equation can be written in the form
\begin{equation*}
2\, \int_{\Rn}\phi(x)\,dx\!\!\int_{\Rn}\frac{|u(y)-u(x)|^{p-2}\bigl(u(y)-u(x)\bigr)
  } {|y-x|^{\alpha
      p}}\,dy + \lambda \int_{\Rn}|u|^{p-2}u\phi\,dx=0
\end{equation*}
provided that the double integral converges.
To see this, split the double integral in (\ref{euler}) into two, one
with $\phi(x)$ and one with   $\phi(y)$. Then use symmetry. This
counts for the factor $2$. By the variational lemma the equation 
\begin{align*}
\mathcal{L}_pu\,(x)&:= 2\,\int_{\Rn}\frac{|u(y)-u(x)|^{p-2}\bigl(u(y)-u(x)\bigr)
  } {|y-x|^{\alpha
      p}}\,dy \nonumber\\
& = -\lambda |u(x)|^{p-2}u(x)
\end{align*}
holds at a.\,e. point $x \in \Omega$, if the inner integral is
summable\footnote{In the linear case this integral operator has been
  treated as the principal value of a singular integral.}. A sufficient condition is that $u$ is Lipschitz continuous
and $\alpha p < p+n-1$ (instead of $< p+n$). In this case
$\mathcal{L}_pu\,(x)$ is continuous in the variable $x$.  In the complement $\Rn \setminus
\Omega$ this equation is not valid, but there we instead have the
information that $u \equiv 0$. Symbolically we can write the
Euler-Lagrange equation as
\begin{equation*}
{\displaystyle \mathcal{L}_pu + \lambda |u|^{p-2}u \, = \, 0.}
\end{equation*}
We remark that if $u \in C_0^{1}(\Rn)$ satisfies
the equation
$$ \mathcal{L}_pu\,(x) + \lambda |u(x)|^{p-2}u(x) \, = \,0$$
in $\Omega$, then 
$$ \int_{\Rn}\!\!\int_{\Rn}\frac{|u(y)-u(x)|^{p-2}\bigl(u(y)-u(x)\bigr)
  \bigl(\phi(y)-\phi(x)\bigr)} {|y-x|^{\alpha
      p}}\,dx\,dy  = \lambda \int_{\Rn}|u|^{p-2}u\phi\,dx$$
holds whenever  $\phi \in C_0^{\infty}(\Omega)$. (See Lemma \ref{pointwiseweak}.)

 Finally, to be on the safe side, we define the concept of
 eigenfunctions. They are weak solutions of the Euler-Lagrange
 equation. Notice that they are defined in the whole space, since we consider them to be extended by zero outside $\Omega$.

\begin{definition} We say that $u\not\equiv 0, u \in W^{s,p}_0(\Omega), \, s = \alpha - n/p$, is
  an \emph{eigenfunction} of $\Omega$, if the Euler-Lagrange equation
  (\ref{euler}) holds for all test functions $\phi \in
  C_0^{\infty}(\Omega)$. The corresponding $\lambda$ is called an
  \emph{eigenvalue}.
\end{definition}

Due to the global nature of the operator $\mathcal{L}_p$ it is not
sufficient to prescribe the boundary values only on the boundary
$\partial \Omega$, but one has to declare that $u = 0$ in the whole
complement
$\Rn \setminus \Omega$. Indeed, a change of $u$ done outside
$\Omega$ can influence the entire operator $\mathcal{L}_pu$.

\section{Viscosity Solutions}

The eigenfunctions were defined as the weak solutions to the
Euler-Lagrange equation in the usual way with test functions under
 the integral sign
(Definition 2). As we will see, they are also viscosity solutions of the equation
$$ \mathcal{L}_pu + \lambda |u|^{p-2}u = 0,$$
provided that they are continuous.
This is another notion. We refer to the book \cite{Koi04} for an introduction. The theory of viscosity solutions is based on \emph{pointwise}
testing: the equation is evaluated  for test functions at points of contact. The
viscosity solutions are assumed to be continuous, 
but the fractional Sobolev space is absent from their definition.

\begin{definition} [Viscosity solutions] \label{pvisc} Suppose that the function $u$
  is continuous in $\Rn$ and that $u \equiv 0$ in $\Rn \setminus
  \Omega$. We say that $u$ is a \emph{viscosity supersolution} in $\Omega$ of the
  equation $$ \mathcal{L}_pu + \lambda |u|^{p-2}u = 0$$ if the
  following holds: whenever $x_0 \in \Omega$ and $\varphi \in
  C_0^1(\Rn)$ are such that
$$ \varphi(x_0) = u(x_0) \quad \text{and} \quad \varphi(x) \leq
u(x)\quad \text{for all} \quad
x \in \Rn,$$
then we have
$$ \mathcal{L}_p\varphi\,(x_0) + \lambda
|\varphi(x_0)|^{p-2}\varphi(x_0) \leq 0.$$
The requirement for a \emph{viscosity subsolution} is symmetric: the test
function is touching from above and the inequality is
reversed. Finally, a \emph{viscosity solution} is defined as being
both a viscosity supersolution and a viscosity subsolution.
\end{definition}

\begin{remark}\label{mono} The required pointwise inequalities for $\varphi$ are
valid also if the function $\varphi(x) + C$ touches $u$  at
$x_0.$ To see that the constant has no influence, use the following
simple monotonicity property for $\psi,\varphi\in C_0^1(\R^n)$:
$$
\textrm{if } \psi \geq \varphi\, \textrm{ and } \psi(x_0) = \varphi(x_0), \textrm{ then }
\mathcal{L}_p\psi\,(x_0) \geq \mathcal{L}_p\varphi\,(x_0).$$
\end{remark}

In order to prove that continuous weak solutions are viscosity
solutions we need a comparison principle.

\begin{lemma}[Comparison Principle]
\label{comparison}
 Let $u$ and $v$ be two continuous functions belonging to
  $W_0^{s,p}(\Rn)$. Let $D \subset\R^n$ be a domain. If
\begin{itemize}
\item   $v \geq u \quad \text{in} \quad \Rn \setminus D,$ and
\item   $ \mathcal{L}_pv\,(x) \leq \mathcal{L}_pu\,(x)$ when $x
  \in D$ in the sense that
\begin{gather*}
 \phantom{>} \int_{\Rn}\!\!\int_{\Rn}\frac{|v(y)-v(x)|^{p-2}\bigl(v(y)-v(x)\bigr)
  \bigl(\phi(y)-\phi(x)\bigr)} {|y-x|^{\alpha
      p}}\,dx\,dy\\  \geq \int_{\Rn}\!\!\int_{\Rn}\frac{|u(y)-u(x)|^{p-2}\bigl(u(y)-u(x)\bigr)
  \bigl(\phi(y)-\phi(x)\bigr)} {|y-x|^{\alpha
      p}}\,dx\,dy  
\end{gather*}
whenever $\phi \in C_0(D),\, \phi \geq 0,$
\end{itemize}
then $v \geq u$ also in $D$. That is,  $v \geq u$ in $\Rn$.
\end{lemma}

\medskip
\begin{proof} Subtract the equations. The resulting integral
 \begin{gather*}
 \int_{\R^n}\!\!\int_{\R^n}\frac{\bigl[|v(y)\!-\!v(x)|^{p-2}\bigl(v(y)\!-\!v(x) \bigr) -|u(y)\!-\!u(x)|^{p-2}\bigl(u(y)\!-\!u(x)\bigr)\bigr]
  \bigl(\phi(y)\!-\!\phi(x)\bigr)} {|y-x|^{\alpha
      p}}dxdy
\end{gather*}
is non-negative if $\phi\geq 0$. We aim at showing that the
integrand is non-positive for the choice $\phi = (u-v)^+$. The identity
$$|b|^{p-2}b-|a|^{p-2}a\, =\, (p-1)(b-a)\int_{0}^{1}|a+t(b-a)|^{p-2}\,dt$$
 with $a = u(y)-u(x)$ and $b = v(y)-v(x)$ gives the formula
\begin{gather*}
|v(y)-v(x)|^{p-2}\bigl(v(y)-v(x) \bigr)
-|u(y)-u(x)|^{p-2}\bigl(u(y)-u(x)\bigr)\\ 
= (p-1)\bigl\{u(x)-v(x) - \bigl(u(y)-v(y)\bigr)\bigr\}Q(x,y),
\end{gather*}
which is to be used in the integrand above. We have abbreviated\footnote{The
  idea is obvious in the case $p=2$.}
$$Q(x,y) =
\int_{0}^{1}\left|\bigl(u(y)-u(x)\bigr)+t\bigl(\bigl(v(y)-v(x)\bigr)
  -\bigl( u(y)-u(x)\bigr)\Bigr)\right|^{p-2}\,dt.
$$
We see that $Q(x,y) \geq 0$, and $Q(x,y) = 0$ only if $v(y) = v(x)$ and
$u(y) = u(x)$. We choose the test function $\phi= (u-v)^+$ and write 
$$ \psi = u-v = (u-v)^+ -(u-v)^-, \quad \phi = (u-v)^+ = \psi^+.$$
The integrand becomes the factor $(p-1)Q(x,y)/\abs{y-x}^{\alpha}$
multiplied with
\begin{align*}
[\psi(x)&-\psi(y)][\phi(y)-\phi(x)]\\
&= [\psi^+(x)-\psi^-(x)-\psi^+(y)+\psi^-(y)][\psi^+(y)-\psi^+(x)]\\
&= - \bigl(\psi^+(y)-\psi^+(x)\bigr)^2 +
\bigl(\psi^-(y)-\psi^-(x)\bigr)\bigl(\psi^+(y)-\psi^+(x)\bigr)\\
&=  -\bigl(\psi^+(y)-\psi^+(x)\bigr)^2
-\psi^-(y)\psi^+(x)-\psi^-(x)\psi^+(y),
\end{align*}
where the formula $\psi^-(x)\psi^+(x) = 0$ was used. The
integrand contains only negative terms and, to avoid a contradiction, it
is necessary that
$$\psi^+(y) = \psi^+(x) \quad \text{or} \quad Q(x,y) = 0$$
at a.\,e. point $(x,y)$. Also the latter alternative implies that $\psi^+(y) =
\psi^+(x)$. In other words, the identity
$$\bigl(u(y)-v(y)\bigr)^+ = \bigl(u(x)-v(x)\bigr)^+$$
must hold. It follows that $u(x)-v(x) = C = $ Constant $ 
\geq 0$ in the set where $u(x)\geq v(x)$. The boundary condition
requires that $C = 0$. The claim $v \geq u$ follows. \end{proof}

\begin{lemma}\label{pointwiseweak}
Let $f \in C(\Omega)$ and $v \in C^1_{0}(\Rn)$. If the inequality
$$\mathcal{L}_pv\,(x) \leq f(x)$$
is valid at each point $x$ in the subdomain $D \subset \Omega$,
then the inequality
\begin{equation}
\label{f}
 -\, \int_{\Rn}\!\!\int_{\Rn}\frac{|v(y)-v(x)|^{p-2}\bigl(v(y)-v(x)\bigr)
  \bigl(\phi(y)-\phi(x)\bigr)} {|y-x|^{\alpha
      p}}\,dx\,dy \leq \int_{\Omega}f(x)\phi(x)\,dx
\end{equation}
holds for all $\phi \in C_0(D),\, \phi \geq 0$.
\end{lemma}

\medskip
\begin{proof} Multiply the inequality $\mathcal{L}_pv\,(x) \leq f(x)$
with $\phi(x)$ and integrate over $D$ to obtain
$$
\phantom{-} 2\,\int_{D}\!\int_{\Rn}\frac{|v(y)-v(x)|^{p-2}\bigl(v(y)-v(x)\bigr)
  \phi(x)} {|y-x|^{\alpha
      p}}\,dy\,dx \leq \int_{\Omega}f(x)\phi(x)\,dx.$$
We can replace $D$ by $\Rn$ in the outer integration. Switching $x$
and $y$, we can write
$$
-2\,\int_{D}\!\int_{\Rn}\frac{|v(y)-v(x)|^{p-2}\bigl(v(y)-v(x)\bigr)
  \phi(y)} {|y-x|^{\alpha
      p}}\,dx\,dy \leq \int_{\Omega}f(y)\phi(y)\,dy.$$
Notice the minus sign. Adding the expressions we arrive at (\ref{f}). \end{proof}

\begin{prop} Let $\alpha p < n+p-1$. An eigenfunction $u \in C_0(\overline{\Omega})$ is a
  viscosity solution of the equation 
$$\mathcal{L}_pu = - \lambda|u|^{p-2}u.$$
\end{prop}

\begin{proof} We prove the case of a subsolution, assuming for
simplicity that $u \geq 0$.  Our proof is indirect. If $u$ is not a
viscosity subsolution, the
\emph{antithesis} is that there exist a testfunction $\phi$ and a point
$x_0$ in $\Omega$
such that
\begin{gather*}
\phi \in C_0^1({\R^n}), \quad \phi \geq u, \quad \phi(x_0)
= u(x_0),\\
\mathcal{L}_p\phi\,(x_0) < - \lambda |\phi(x_0)|^{p-2}\phi(x_0).
\end{gather*}
By continuity
\begin{equation*}
\mathcal{L}_p\phi\,(x) < - \lambda |\phi(x_0)|^{p-2}\phi(x_0)
\end{equation*}
holds when $x \in B(x_0,2r)$, where the radius $r$ is small
enough. This means that $\phi$ is a ``strict supersolution'' in the
ball. 
 We need to modify $\phi$. For the purpose we choose a smooth radial
 function $\eta \in C^{\infty}(\Rn)$  such that $0 \leq \eta(x) \leq
 1$
and
\begin{align*}
&\eta(x_0) = 0,\\
&\eta(x )> 0, \,\,\text{when}\,\, x \not = x_0,\\
&\eta(x) = 1, \,\,\text{when}\,\, |x-x_0| \geq r.
\end{align*}
Let $\varepsilon > 0$ be  small and consider the
function
$$v = v_{\varepsilon} = \phi + \varepsilon \eta - \varepsilon.$$
Outside $B(x_0,r)$ it coincides with $\phi$. By Lebesgue's Dominated
Convergence Theorem
$$
\lim_{\varepsilon \to 0}\mathcal{L}_pv_{\varepsilon}\,(x) = \mathcal{L}_p\phi\,(x).
$$
A closer inspection reveals that, actually, the limit is uniform
 on compact sets. Since $u$ is continuous, it follows that for
a sufficiently small $\varepsilon > 0$
$$\mathcal{L}_pv_{\varepsilon}\,(x) < - \lambda |u(x)|^{p-2}u(x) =
f(x)$$
when $x \in B(x_0,r)$. By the previous lemma this inequality also holds in the
weak sense with test functions under the integral sign. Thus equation
(\ref{f})
is available. 

Now $\mathcal{L}_pv \leq \mathcal{L}_pu$ in the weak sense in
$B(x_0,r)$, as described in Lemma \ref{comparison}. By the construction
\begin{align*}
v = \phi \quad \text{in} \quad \Rn \setminus B(x_0,r).
\end{align*}
In particular,
$$ v \geq u \quad \text{in} \quad \Rn \setminus  B(x_0,r).$$ By the comparison
principle (Lemma \ref{comparison}) 
$$ v \geq u \quad \text{in} \quad   B(x_0,r).$$ But this contradicts
the fact that
$$v(x_0) = \phi(x_0) - \varepsilon = u(x_0) - \varepsilon < u(x_0).$$
Thus the antithesis is false. We have proved that $u$ is a viscosity
subsolution. ---The case of viscosity supersolutions is
similar. \end{proof}

\medskip
The next result shows that the first
eigenfunctions cannot have zeros in the domain.

\begin{lemma}[Positivity]
\label{positivity}
Assume $u \geq 0$ and $u\equiv 0$ in $\R^n\setminus \Omega$. If $u$ is a viscosity supersolution of the equation
  $\mathcal{L}_pu = 0$ in $\Omega$, then either $u>0$ in $\Omega$
  or $u\equiv0.$
\end{lemma}

\begin{proof} Recall that being a supersolution means that $\mathcal{L}_p\psi \leq 0$
for the test functions below.  At a point $x_0$ in $\Omega$ where $u(x_0) = 0$ we have
for any test function $\psi$ that touches $u$ from below  that
$$0 \geq  \mathcal{L}_p\,\psi(x_0) = 2
\int_{\Rn}\frac{|\psi(y)|^{p-2}\psi(y)\,dy}{|y-x_0|^{\alpha p}}$$
since $\psi(x_0) = 0$. If $\psi \geq 0$ this implies that $\psi \equiv
0$. But, if $u \not \equiv 0$, we can certainly, using the continuity
of $u$, select a test function $\psi$
so that
$0 \leq \psi \leq u$ which is positive at some point.
\end{proof}

 It is noteworthy that the result above does not hold true if $u$ is
 not non-negative in $\R^n\setminus \Omega$. This is related to the
 fact that the usual Harnack inequality fails for non-local operators
 in general. See \cite{Kas07}, for an explicit counter example in the
 case $p=2$.

\section{Uniqueness of Positive Eigenfunctions}

We know that a continuous non-negative eigenfunction cannot have any zeros in
the domain $\Omega$ (Lemma \ref{positivity}). We shall prove that the
only positive eigenfunctions are the first ones and also that \emph{the first
eigenvalue is simple}. In other words, if $u_1$ is a minimizer of the 
Rayleigh quotient, \emph{all positive eigenfunctions are
of the form} $u(x) = Cu_1(x)$. First, we have to prove that the
minimizer is unique, except for multiplication by constants. Then it
will be established that a positive eigenfunction is a
minimizer. ---We will encounter the difficulty with the lack of an
adequate regularity theory for our equation. To avoid such issues
here, we deliberately take $\alpha p > 2n$, which guarantees the
continuity of the eigenfunctions.

We use an elementary inequality for the auxiliary function $$\ss(s,t) =
|s^{1/p}-t^{1/p}|^p, \qquad s>0,\, t>0.$$

\begin{lemma} The function $\ss(s,t)$ is convex in the quadrant
  $s>0,t>0$. Thus
$$\ss\bigl(\frac{s_1+s_2}{2}, \frac{t_1+t_2}{2}\bigr) \leq
\frac{1}{2}\ss(s_1,t_1) + \frac{1}{2}\ss(s_2,t_2).$$
Moreover, equality holds only for $s_1t_2 = s_2t_1.$
\end{lemma}
\medskip
\begin{proof} As a matter of fact, $\ss$ is a solution to the
Monge-Amp\`ere equation $$\ss_{ss}\ss_{tt}-\ss_{st}^2 = 0.$$ A direct
calculation yields the expression 
$$\ss_{ss}(s,t)X^{2}+2\ss_{st}(s,t)XY+\ss_{tt}(s,t)Y^{2} =
\tfrac{p-1}{p}|s^{1/p}-t^{1/p}|^{p-2}(st)^{1/p}\Bigl(\frac{s}{X}-\frac{t}{Y}\Bigr)^2$$
for the quadratic form associated with the Hessian matrix. The
quadratic form is strictly positive except when $s=t$ or
$\frac{s}{X}=\frac{t}{Y}$. The result follows by inspection. \end{proof}

\begin{thm} Take $\alpha p > 2n$. The minimizer of the Rayleigh quotient is unique,
  except that it may be multiplied by a constant.
\end{thm}

\begin{proof} Our proof is a modification of the proof given in \cite{BK02}. If $u$ and $v$ are minimizers, so are $|u|$ and
$|v|$. Since $|u| > 0$ and $|v| > 0$ in $\Omega$ by Lemma \ref{positivity}, we may by
continuity assume that $u>0$ and $v >0$ from the beginning. Our claim
is that $u(x) = C v(x)$.

Normalize the functions so that
$$\int_{\Rn}u^p\,dx =  \int_{\Rn}v^p\,dx = 1$$
and consider the admissible function
$$w = \Bigl(\frac{u^p+v^p}{2}\Bigr)^{1/p}$$
in the Rayleigh quotient. Also
$$\int_{\Rn}w^p\,dx = 1$$
by construction. In the numerator we have, according to the previous
lemma,
\begin{equation}
\label{lik}
|w(y)-w(x)|^p \leq \frac{1}{2}|u(y)-u(x)|^p+
\frac{1}{2}|v(y)-v(x)|^p\end{equation}
with equality \emph{only} for
\begin{equation}
\label{likk}
u(x)v(y)=u(y)v(x).
\end{equation}
Divide by $|y-x|^{\alpha p}$, integrate, and use the normalization to
conclude that
\begin{align*}
\lambda_1 &\leq \frac{\displaystyle \int_{\Rn}\!\!
  \int_{\Rn}\frac{|w(y)-w(x)|^p}{|y-x|^{\alpha
      p}}\,dx\,dy}{\displaystyle  \int_{\Rn}w^p\,dx}\\
 &\leq \frac{\displaystyle \frac{1}{2}  \int_{\Rn}\!\!
  \int_{\Rn}\frac{|u(y)-u(x)|^p}{|y-x|^{\alpha
      p}}\,dx\,dy}{\displaystyle  \int_{\Rn}u^p\,dx}
+  \frac{\displaystyle\frac{1}{2}   \int_{\Rn}\!\!
  \int_{\Rn}\frac{|v(y)-v(x)|^p}{|y-x|^{\alpha
      p}}\,dx\,dy}{\displaystyle  \int_{\Rn}v^p\,dx}\\
& = \frac{1}{2}\lambda_1+
\frac{1}{2}\lambda_1  =\lambda_1.
\end{align*}
Thus the only possibility is that equality holds in (\ref{lik}) for $x$
and $y$ in $\Omega$. Thus (\ref{likk}) holds, which proves that $u(x)
= C v(x)$.
\end{proof}

\begin{lemma} [Exhaustion]
\label{exhaustion}
 Let 
$$\Omega_1 \subset  \Omega_2 \subset \Omega_3 \subset \cdots \subset\Omega,\qquad \Omega = \bigcup \Omega_j.$$ Then
$$\lim_{j \to \infty} \lambda_1(\Omega_j) = \lambda_1(\Omega).$$
\end{lemma}

\medskip

\begin{proof} Since $\lambda_1(\Omega_1) \geq  \lambda_1(\Omega_2)
\geq \cdots \geq \lambda_1(\Omega)$ the limit exists. Given $\varepsilon > 0$,
there exists a $\phi \in C_0^{\infty}(\Omega)$ such that
$$  \frac{\displaystyle \int_{\Rn}\!\!
  \int_{\Rn}\frac{|\phi(y)-\phi(x)|^p}{|y-x|^{\alpha
      p}}\,dx\,dy}{\displaystyle  \int_{\Rn}|\phi|^p\,dx} <
\lambda_1(\Omega) +\varepsilon,$$
because $\lambda_1(\Omega)$ is the infimum. For $j$ large enough,
$\spt(\phi) \subset \Omega_j$ and thus $\phi$ will do as test
function in the Rayleigh quotient also for the subdomain $\Omega_j$. It
follows that 
$$\lambda_1(\Omega_j) <  \lambda_1(\Omega) +\varepsilon$$
for sufficiently large $j$. \end{proof}

\medskip

Any domain $\Omega$ can be exhausted by a sequence of
\emph{smooth} domains $\Omega_j\subset\subset \Omega$. See for example \cite[p.~317-319]{Kel67}.

\begin{thm}
\label{signchange}
 Take $\alpha p>2n$. Then a non-negative eigenfunction minimizes the Rayleigh
  quotient.
\end{thm}

\medskip

\begin{proof} The proof is based on a construction in \cite{OT88}; see also \cite{KL06}.

Antithesis: Assume that $v \geq 0$ is a weak solution in
$\Omega$ of the Euler-Lagrange equation (\ref{euler})
with eigenvalue $\lambda > \lambda_1(\Omega)$.

 By Theorem \ref{ca} $v$ is continuous. As $v \not \equiv 0$ we
have that $v > 0$   by Lemma \ref{positivity}. According to Lemma
\ref{exhaustion} and the remark following it, we can construct a smooth
domain $\Omega^{*} \subset \subset \Omega$ such that also 
$$ \lambda_{1}^{*} =  \lambda_{1}(\Omega^{*}) < \lambda.$$
Let $v^{*}$ denote the first eigenfunction in  $\Omega^{*}$; its
eigenvalue is $ \lambda_{1}^{*}$. Since $\alpha p > 2n$,\, $v^{*} \in C(\overline{ \Omega^{*}})$ and $v^{*} = 0$ on 
$\partial \Omega^{*}$ and in $\R^n\setminus\Omega^*$. Because $v > 0$ in $\Omega$,
$$\min_{\overline{ \Omega^{*}}} v \,> \, 0,$$
and we can arrange it so that 
$$v \geq v^{*}\quad \text{in}\quad \Rn$$
by multiplying $v$ by a suitable constant, if needed.

Let $\phi \in C_0^{\infty}(\Omega^{*}),\, \phi \geq 0$, be a test
function. Then the equations are
\begin{align*}
&\int_{\Rn}\!\!\int_{\Rn}\frac{|v^{*}(y)-v^{*}(x)|^{p-2}\bigl(v^{*}(y)-v^{*}(x)\bigr)
  \bigl(\phi(y)-\phi(x)\bigr)} {|y-x|^{\alpha
      p}}\,dx\,dy\\
& = \lambda_1^{*} \int_{\Rn}v^{*}(y)^{p-1}\phi(y)\,dy \leq
 \lambda_1^{*} \int_{\Rn}v(y)^{p-1}\phi(y)\,dy =  \lambda
 \int_{\Rn}(\varkappa v(y))^{p-1}\phi(y)\,dy\\
& =
\int_{\Rn}\!\!\int_{\Rn}\frac{|\varkappa v(y)-\varkappa
  v(x)|^{p-2}\bigl(\varkappa v(y)- \varkappa v(x)\bigr)
  \bigl(\phi(y)-\phi(x)\bigr)} {|y-x|^{\alpha
      p}}\,dx\,dy,
\end{align*}
where we have denoted
$$\varkappa = \Bigl(\frac{\lambda_1^{*}}{\lambda}\Bigr)^{1/(p-1)}\, <
\, 1.$$
Symbolically, $\mathcal{L}_pv^* \geq \mathcal{L}_p(\varkappa v)$ in
$\Omega^*$ and $\varkappa v \geq v^*$ in $\Rn \setminus \Omega$. The
Comparison Principle (Lemma \ref{comparison}) yields that
$$\varkappa v \geq v^*. \qquad(0 < \varkappa < 1)$$ 
We can repeate the procedure, now starting with the function
$\varkappa v$ in the place of $v$. This yields $\varkappa(\varkappa v)
\geq v^*$. By iteration we arrive at 
$$\varkappa^{j} v\geq v^*, \quad j = 1,2,\ldots$$
When $\varkappa^{j} \to 0$ as $j \to \infty$ we obtain the
contradiction that $v^* \equiv 0$.  \end{proof}

\section{Higher Eigenvalues}

For a fixed exponent $p$ the  set of all eigenvalues form the
\emph{spectrum} $\{\lambda \}.$ By compactness arguments \emph{the spectrum is a closed set}.
The higher eigenvalues are associated with sign-changing
eigenfunctions.
It is well-known that, for a differential operator like the ordinary
Laplacian for instance, a restriction of a higher eigenfunction to
one of its nodal domains is a \emph{first} eigenfunction with respect
to that subdomain. Then a higher eigenvalue of a domain is a first
eigenvalue for any nodal domain. This property holds for many other equations,
too. However, we encounter a new phenomenon for our operator. The
non-local nature of the problem causes the higher eigenvalues to be too large for this property to hold.

Let us begin by recalling that, given an eigenfunction, its nodal
domains are the connected open components of the sets $\{u >0\}$ and
$\{u<0\}$. In passing, we mention that also the quantities $ \lambda_1(\{u
>0\})$ and $ \lambda_1(\{u<0\})$ can be defined in the natural way,
although the open sets involved are not always connected ones.

\begin{thm} \label{toobig}If $u$ is a continuous sign changing eigenfunction with eigenvalue
  $\lambda(\Omega)$, then the strict inequalities 
$$\lambda(\Omega) > \lambda_1(\Omega^+)\quad \text{and} \quad
\lambda(\Omega) > \lambda_1(\Omega^-),$$
 hold for the  open sets   $\Omega^+ = \{u>0\}$ and $\Omega^- =
\{u<0\}$. Moreover,
$$\lambda \geq C(n,p,\alpha)\,|\Omega^+|^{-\frac{\alpha p -n}{n}}
\quad\text{and}\quad \lambda \geq  C(n,p,\alpha)\,|\Omega^-|^{-\frac{\alpha
    p -n}{n}}
.$$
\end{thm}

\medskip

\begin{proof} Let $u = u^+-u^-$ be the usual decomposition where $u^+\geq0,\,u^-\geq0$.
Choose the test function $\phi = u^+$ in the Euler-Lagrange equation
(\ref{euler}). We need to have command over the sign
of the product
\begin{align*}
[u(y)&-u(x)][\phi(y)-\phi(x)]\\
&= [u^{+}(y)-u^{+}(x)]^{2}
-\bigl(u^{-}(y)-u^{-}(x)\bigr)\bigl(u^{+}(y)-u^{+}(x)\bigr)\\
&=  [u^{+}(y)-u^{+}(x)]^{2} + u^{+}(y)u^{-}(x)+u^{+}(x)u^{-}(y),
\end{align*}
where it was used that $u^{+}(x)u^{-}(x) = 0$. The Euler-Lagrange equation becomes
\begin{align*}
\lambda \int_{\Omega}|u^{+}|^p\,dx\,= &\int_{\Rn}\!\!\int_{\Rn}\frac{|u(y)-u(x)|^{p-2}\bigl(u^+(y)-u^+(x)\bigr)^{2}} {|y-x|^{\alpha
      p}}\,dx\,dy \nonumber\\
 +
2\,&\int_{\Rn}\!\!\int_{\Rn}\frac{|u(y)-u(x)|^{p-2}u^+(y)u^-(x))} {|y-x|^{\alpha
      p}}\,dx\,dy.
\end{align*}
The formula
\begin{gather*}
|u(y)-u(x)|^{2} = \bigl(u^{+}(y)-u^{+}(x)\bigr)^{2} +
\bigl(u^{-}(y)-u^{-}(x)\bigr)^{2} \\
-2 \bigl(u^{+}(y)-u^{+}(x)\bigr)\bigl(u^{-}(y)-u^{-}(x)\bigr)\\
=  \bigl(u^{+}(y)-u^{+}(x)\bigr)^{2} +
\bigl(u^{-}(y)-u^{-}(x)\bigr)^{2} +2u^{+}(y)u^{-}(x)+ 2
u^{+}(x)u^{-}(y)
\end{gather*}
implies the estimate
\begin{align*}
\lambda \int_{\Omega^{+}}|u^{+}|^p\,dx\,\geq &\int_{\Rn}\!\!\int_{\Rn}\frac{|u^{+}(y)-u^{+}(x)|^{p}} {|y-x|^{\alpha
      p}}\,dx\,dy \\
 +
2^{p/2}\,&\int_{\Rn}\!\!\int_{\Rn}\frac{\bigl(u^+(y)u^-(x)\bigr)^{\frac{p}{2}}} {|y-x|^{\alpha
      p}}\,dx\,dy.
\end{align*}
It follows that
\begin{equation}
\label{excess}
\lambda \geq \lambda_{1}(\Omega^{+}) + 2^{p/2} \frac{\displaystyle \int_{\Rn}\!\!\int_{\Rn}\frac{\bigl(u^+(y)u^-(x)\bigr)^{\frac{p}{2}}} {|y-x|^{\alpha
      p}}\,dx\,dy}{\displaystyle\int_{\Omega^{+}}|u^{+}|^p\,dx},
\end{equation}
because $u^+$ is admissible in the Rayleigh quotient as test function
for $\Omega^+$. This clearly shows that we have the \emph{strict}
inequality $\lambda > \lambda_1(\Omega^+)$.

By inequality (\ref{sobomes}) it follows immediately that
$$\lambda \int_{\Omega^{+}}|u^{+}|^p\,dx\,\geq \int_{\Rn}\!\!\int_{\Rn}\frac{|u^{+}(y)-u^{+}(x)|^{p}} {|y-x|^{\alpha
      p}}\,dx\,dy \geq C\,|\Omega^+|^{-\frac{\alpha p -
      n}{n}}\int_{\Omega^{+}}|u^{+}|^p\,dx,$$ 
and so, upon division, $\lambda \geq  C\,|\Omega^+|^{-\frac{\alpha p -
      n}{n}}$.
 ---The proof for
$\Omega^-$ is symmetric. \end{proof}

\medskip

\begin{remark} The excess term in (\ref{excess}) can be improved a
little, but it is not evident, whether one can get a bound free
of the functions $u^+$ and $u^-$.
\end{remark}

Due to the fact that higher eigenfunctions are sign-changing,
there is a gap in the spectrum just above the first eigenvalue
$\lambda_1$. Consequently, the second eigenvalue is well defined as
the number
$$\lambda_2 = \inf\{\lambda > \lambda_1\}.$$
The minimum is attained. (See \cite{Ana87} for the local case.)

\begin{thm} 
Take $\alpha p>2n$. Then the first eigenvalue is isolated.
\end{thm}

\begin{proof}
Suppose that there is a sequence of eigenvalues $\lambda_k^{'}$
tending to $\lambda_1$, $\lambda_k'\neq \lambda_1$. If $u_k$ denotes the corresponding normalized
eigenfunction, we have
$$ \int_{\Omega}|u_k|^p\,dx\,=\,1,\quad \lambda_k^{'} =  \int_{\Rn}\!\!\int_{\Rn}\frac{|u_k(y)-u_k(x)|^p}{|y-x|^{\alpha
      p}}\,dx\,dy .$$
By compactness (cf. Theorem \ref{cpt}) we can construct a subsequence and a function
$u \in W_0^{s,p}(\Omega),\, s = \alpha - n/p$, such that 
$$u_{k_{j}} \to u \quad \text{in} \quad L^p(\Rn).$$
Extracting a further subsequence we can assume that $\lim u_{k_{j}}(x)
= u(x)$ a.\,e.. By Fatou's lemma 
$$\frac{\displaystyle \int_{\Rn}\!\!\int_{\Rn}\frac{|u(y)-u(x)|^p}{|y-x|^{\alpha
      p}}\,dx\,dy}{\displaystyle \int_{\Omega}|u(x)|^p\,dx}\quad \leq\quad \lim_{j
  \to \infty}\lambda^{'}_{k_{j}}  = \lambda_1.$$
We read off that $u$ is a minimizer and therefore the first
eigenfunction. From Lemma \ref{positivity}, either $u > 0$ in $\Omega$ or  $u < 0$ in $\Omega$. But
if $ \lambda_k^{'} > \lambda_1$ then $u_k$ must change signs in $\Omega$
in view of Theorem \ref{signchange}. Both sets
$$\Omega_k^{+} = \{u_k >0\}\quad \text{and}\quad \Omega_k^{-} = \{u_k
<0\}$$
are non-empty and their measures cannot tend to zero, because small
subdomains have large eigenvalues. Indeed, by Theorem \ref{toobig}
\begin{align*}
\lambda_k^{'}&  \geq \lambda_1(\Omega_k^{+}) \geq
C\,|\Omega_k^{+}|^{1-\alpha p/n}, \\
\lambda_k^{'}&  \geq \lambda_1(\Omega_k^{-}) \geq
C\,|\Omega_k^{-}|^{1-\alpha p/n}.
\end{align*}
Both sets
$$\Omega^+ = \limsup \Omega^{+}_{k_{j}}, \quad \Omega^{-} = \limsup
\Omega^{-}_{k_{j}}, $$
have positive measure by a selection  procedure. Passing to a suitable
subsequence we can show that $u\geq 0$ in $\Omega^+$ and $u\leq 0$ in
$\Omega^-$. This is never possible for a first eigenfunction.
\end{proof}

\medskip

\medskip

\section{Passage to Infinity}\label{sec:passage}

In order to study the asymptotic case  $p\to \infty$ we fix $\alpha$
so that
$$0 < \alpha \leq 1$$
and regard $p$ as sufficiently large, say $\alpha p> 2n$. Taking the
$p^{\text{th}}$ root of the Rayleigh quotient and sending $p\to \infty$ we formally arrive at the minimization problem
\begin{equation}
\label{minimum}
\inf_{\phi}\,
\frac{\left\|\frac{\phi(y)-\phi(x)}{|y-x|^{\alpha}}\right\|_{L^{\infty}(\Rn
    \times \Rn)}}{\|\phi\|_{L^{\infty}(\Rn)} }\,=\, \inf_{\phi}\,
\frac{\left\|\frac{\phi(y)-\phi(x)}{|y-x|^{\alpha}}\right\|_{L^{\infty}(\Omega
    \times \Omega)}}{\|\phi\|_{L^{\infty}(\Omega)} }   \quad =\quad
    \Lambda_{\infty}^{\alpha},
\end{equation}
where the infimum is taken over all $\phi \in
C_0^{\infty}(\Omega)$. It will turn out that 
$$ \Lambda_{\infty}^{\alpha}= \Bigl(\frac{1}{\underset{x\in
      \Omega}{\max}\dist(x,\Rn\setminus\Omega)}\Bigr)^{\alpha},$$
so that the notation is consistent with $ \Lambda_{\infty}^{\alpha} =
\bigl(\Lambda_{\infty}\bigr)^{\alpha}$. It is clear that the infimum
is the same if all points outside $\Omega$ are ignored. \emph{The
  minimum is always attained}, but in the larger space $W^{1,\infty}_0(\Omega)$. Indeed, let $B(x_0,R)$ be the largest
open ball contained in $\Omega$. (There may be several such
balls). Then the function $$\phi(x) = [R - |x-x_0|]^+$$
solves the minimization problem and yields $ \Lambda_{\infty}^{\alpha}
= R^{-\alpha}$.  To rigorously prove the lower bound for an arbitrary competing $\phi\in C_0^\infty(\Omega)$, we notice that if $\xi$ is the closest
boundary point from $x$ then 
$$\phi(x) =  \phi(x)- \phi(\xi) = |x-\xi|^{\alpha
       } \left|\frac{\phi(x)-\phi(\xi)}{|x-\xi|^{\alpha
      }}\right| \leq  |x-\xi|^{\alpha}|[\phi]|_{\alpha} =
  \delta(x)^{\alpha}|[\phi]|_{\alpha}.$$
Recall the notation 
$$|[\phi]|_{\alpha} = 
{\left\|\frac{\phi(y)-\phi(x)}{|y-x|^{\alpha}}\right\|_{L^{\infty}(\Rn
    \times \Rn)}}, \quad \delta(x) =
\dist(x,\Rn\setminus \Omega).$$
Now $ \delta(x) \leq R$ and consequently $\|\phi\|_\infty\leq
R^{\alpha}|[\phi]|_{\alpha}$. It follows that
$$\frac{1}{R^{\alpha}} \leq \frac{|[\phi]|_{\alpha}}{\|\phi\|_\infty}  ,$$
as desired. The calculations showing that this minimum is attained can
be found in the proof of the next proposition.

Setting $\lambda_p$ equal to the
\emph{first} eigenvalue, the following limit is easy to establish.

\begin{prop}
\label{limlambda} We have
$$\lim_{p \to \infty}\sqrt [p]{\lambda_{p}} = \frac{1}{R^{\alpha}},$$
where $R = \max\{\dist(x,\R^n\setminus \Omega)\}$ is the radius of the
largest inscribed ball in the domain $\Omega$.
\end{prop}

\begin{proof} Let $\phi$ be a test function so that
 \begin{equation*}
\lambda_{p} \leq
\frac{\displaystyle \int_{\Rn}\!\!\int_{\Rn}\frac{|\phi(y)-\phi(x)|^p}{|y-x|^{\alpha
      p}}\,dx\,dy}{\displaystyle
\int_{\Rn}|\phi(x)|^p\,dx}.
\end{equation*}
Taking the $p^{\text{th}}$ root and letting $p \to \infty$ we obtain
the bound
$$\limsup_{p \to \infty} \lambda_{p}^{\frac{1}{p}}  \leq
\frac{\displaystyle \left \|\frac{\phi(y)-\phi(x)}{|y-x|^{\alpha
      }}\right\|_{L^{\infty}(\Rn\!\times\Rn)}} {\displaystyle
  \|\phi(x)\|_{L^{\infty}(\Rn)}}.$$
As $\phi$ we take the distance function $\delta = \delta(x) =
[R-|x-x_0|]^{+}$ for the inscribed ball, the center of which we may assume
to be $x_0 = 0$. Then $\|\phi\|_\infty = R$ and a direct computation gives
$$\left|\frac{\delta(y)-\delta(x)}{|y-x|^{\alpha}}\right| =
\frac{\bigl|\,|y|
-|x|\,\bigr|}{|y-x|^{\alpha}},$$
from which the desired upper bound follows by calculus.
  
To get the lower bound, we select an increasing sequence $p_j \to
\infty$ such that $\lim \lambda_{p_{j}}^{1/p_{j}} = \liminf
\lambda_p^{1/p}$. Let $u_{p_{j}}$ be the corresponding minimizer of
the Rayleigh quotient normalized so that
$$\int_{\Omega}u_{p_{j}}^{p_j}\,dx = 1,\quad  \lambda_{p_{j}} =
\int_{\Rn}\!\!\int_{\Rn} \left|\frac{u_{p_{j}}(y)-u_{p_{j}}(x)}{|y-x|^{\alpha
      }}\right|^{p_j}\,dx\,dy$$
By the inclusion in H\"older spaces, Theorem \ref{ca}, a subsequence converges uniformly in $\Rn$ to a function $u \in
C_0(\overline{\Omega}).$ In particular the normalization is preserved: $\|u\|_{L^{\infty}(\Omega)} =
1.$ In order to avoid an unbounded domain in H\"{o}lder's inequality below,
we integrate first only over $\Omega \times \Omega$. For a fixed
exponent $q$ Fatou's lemma and  H\"{o}lder's inequality imply
\begin{align*}
\int_{\Omega}\!\int_{\Omega}& \left|\frac{u(y)-u(x)}{|y-x|^{\alpha
      }}\right|^{q}\,dx\,dy \\
&\leq \liminf_{j\to\infty} \int_{\Omega}\!\int_{\Omega} \left|\frac{u_{p_{j}}(y)-u_{p_{j}}(x)}{|y-x|^{\alpha
      }}\right|^{q} \,dx\,dy\\
&\leq \liminf_{j\to\infty}\, |\Omega|^{2(1-\frac{q}{p_j})}\left\{\int_{\Omega}\!\int_{\Omega} \left|\frac{u_{p_{j}}(y)-u_{p_{j}}(x)}{|y-x|^{\alpha
      }}\right|^{p_j}dx\,dy\right\}^{\frac{q}{p_j}} \\
&\leq |\Omega|^{2}  \liminf_{j\to\infty} \left\{\int_{\Rn}\!\!\int_{\Rn} \left|\frac{u_{p_{j}}(y)-u_{p_{j}}(x)}{|y-x|^{\alpha
      }}\right|^{p_j}dx\,dy\right\}^{\frac{q}{p_j}}\\
& =
|\Omega|^{2}\left(\lim_{j\to\infty}  \lambda_{p_{j}}^{1/p_{j}}\right)^{q}.
\end{align*}
Taking the $q^{\text{th}}$ root of the estimate, then sending $q \to
\infty$ and recalling  the normalization, we see that the minimum is
less than $\underset{j\to\infty}{\liminf}
\lambda_p^{1/p}.$ \end{proof}

\section{The Infinity Euler-Lagrange Equation}

The minimization problem (\ref{minimum}) often has too many solutions,
because a minimizer can be rather freely modified outside the largest
inscribed ball in the domain. To eliminate the ``false solutions'' we
need the limit equation to which the Euler-Lagrange equations tend as
$p \to \infty$. The operator
\begin{equation*}
\mathcal{L}_{\infty}u\,(x) =
\underbrace{\underset{y\in\Rn}{\sup}\,\frac{u(y)-u(x)}{|y-x|^{\alpha}}}_{
 \mathcal{L}_{\infty}^{+}u\,(x)}  \, \,+\,\,\underbrace{\underset{y\in\Rn}{\inf}\,\frac{u(y)-u(x)}{|y-x|^{\alpha}}}_{
 \mathcal{L}_{\infty}^{-}u\,(x)}
\end{equation*}
is fundamental. The decomposition 
$$\mathcal{L}_{\infty}u\,(x) =  \mathcal{L}_{\infty}^{+}u\,(x) +
\mathcal{L}_{\infty}^{-}u\,(x)$$
is not the ordinary  one into positive and negative parts. For
positive solutions we will
derive the limit equation
\begin{equation}
\label{limeq}
{\displaystyle
  \max\left\{\mathcal{L}_{\infty}u\,(x),\,\mathcal{L}_{\infty}^{-}u\,(x)
    + \Lambda_{\infty}^{\alpha}u(x)\right\}\,=\,0},
\end{equation}
and for lack of a better name we refer to this equation as the \emph{$\infty$-eigenvalue equation}. This ``Euler-Lagrange equation'' has to be interpreted in the viscosity sense. The notation above indicates that at each point the
largest of two numbers is zero.

\begin{definition}\label{infeigen} We say that a non-negative function $u \in C_0(\Rn)$
  is a \emph{viscosity supersolution} of the equation
$$\max\left\{\mathcal{L}_{\infty}u\,(x),\,\mathcal{L}_{\infty}^{-}u\,(x)
    + \Lambda_{\infty}^{\alpha}u(x)\right\}\,=\,0$$
in the domain $\Omega$ if the conditions
$$\mathcal{L}_{\infty}\phi\,(x_0)\leq 0 \quad\text{and}\quad \mathcal{L}_{\infty}^{-}\phi\,(x_0)
    + \Lambda_{\infty}^{\alpha}\phi(x_0) \leq 0$$
hold, whenever the test function $\phi \in C_0^{1}(\Rn)$ touches $u$
from below at the point $x_0\in \Omega$. 

We say that $u\in C_0(\Rn)$ is a \emph{viscosity subsolution} if one of
the conditions
$$\mathcal{L}_{\infty}\psi\,(x_0)\geq 0 \quad\text{or}\quad \mathcal{L}_{\infty}^{-}\psi\,(x_0)
    + \Lambda_{\infty}^{\alpha}\psi(x_0) \geq 0$$
holds, whenever the test function
$\psi \in C_0^{1}(\Rn)$ touches $u$
from above at the point $x_0\in \Omega$.

Finally, $u$ is a \emph{viscosity solution} if it is both a viscosity
supersolution and a viscosity subsolution.

A viscosity solution $u\in C_0(\overline \Omega)$, $u>0$, is called a first $\infty$-eigenfunction.
\end{definition}

We consider an arbitrary sequence of first eigenvalues $\lambda_p$
with $p \to \infty$ and denote the corresponding eigenfunction by
$u_p$. The limit procedure requires the following lemma.

\begin{lemma}[Positivity]
\label{pos}
Let $v \in C_0(\Rn)$ be a viscosity supersolution of the
equation $\mathcal{L}_{\infty}v = 0$ in $\Omega$. If $v\geq 0$ in $\R^n$,  then, either $v>0$
in $\Omega$ or $v \equiv 0$.
\end{lemma}

\begin{proof} The concept means that $\mathcal{L}_{\infty}\phi\,(x_0)
\leq 0$ for all test functions touching $v$ from below at a given
point $x_0$ in $\Omega$. Assume now that $v(x_0) = 0$ at some
point. Then there is certainly a test function such that $0 \leq \phi
\leq v$ and $\phi(x_0) = v(x_0) = 0.$ Hence
\begin{align*}
0 \geq \mathcal{L}_{\infty}\phi\,(x_0) =
\max_{\R^n}\,\frac{\phi(y)}{|x_0-y|^{\alpha}}\, +\, \min_{\R^n}\,\frac{\phi(y)}{|x_0-y|^{\alpha}}\geq 
\,\frac{\phi(y)}{|x_0-y|^{\alpha}},
\end{align*}
which implies that $\phi \equiv 0$. As in the proof of Lemma
\ref{positivity} we conclude that $v \equiv 0$. \end{proof}

As we know $\sqrt[p]{\Lambda_p} \to
\Lambda_{\infty}^{\alpha}$, by Proposition \ref{limlambda}. For the eigenfunctions we have to go to subsequences.

\begin{thm}
\label{convergence} There exists a subsequence of functions $u_p$
 converging uniformly in
$\Omega$ to a function $u \in C_0(\overline{\Omega})$ which is a
viscosity solution in $\Omega$ of the $\infty$-eigenvalue equation
(\ref{limeq}).
\end{thm}

\begin{proof} If we normalize the functions so that $\|u_p\|_{L^{p}} =
1$, then for $sp = \alpha p - n$
$$\|u_p\|_{W^{s,p}(\R^n)} \leq C(1+\sqrt[p]{\lambda_p}).$$ Since
$\sqrt[p]{\lambda_p }\to R^{-\alpha}$ we have a bound independent of
$p$. For an arbitrary $\gamma\in (0,\alpha)$, we have a bound on the
H\"{o}lder norms $\|u_p\|_{C^{\gamma}(\Rn)}$
for large $p$'s, according to Theorem \ref{ca}. By Ascoli's theorem
we can extract a subsequence $u_j = u_{p_{j}}$ that converges
uniformly in each $C^{\gamma}(\Rn)$ to a function $u$. It follows that
$u \in C_0(\overline{\Omega})$ and $u = 0$ in $\Rn\setminus \Omega$. 

\emph{Viscosity Supersolution.} In order to prove that the limit
function is a viscosity supersolution in $\Omega$, we assume that
$\phi$ is a test function touching $u$ from below at a point $x_0$. We
may assume that the touching is strict by
considering $\phi(x) - |x|^2\eta(x)$, where $\eta \in
C_0^{\infty}(\Rn)$ is a function such that $\eta = 1$ in a
neighbourhood of $x_0$ and $\eta \geq 0$.
We can assure that $u_j -\phi$ assumes its minimum at points $x_j \to
x_0$. This is standard reasoning. By adding a suitable constant $c_j$
 we can arrange it so that $\phi + c_j$ touches $u_j$ from below at
 the point $x_j$. Recall that the constant has no influence in the testing procedure according
 to Remark \ref{mono}.

Since an eigenfunction is a viscosity solution, we have the inequality
$$\mathcal{L}_{p_j}\phi\,(x_j) + \lambda_{p_{j}}u_j^{p_j-1}(x_j) \leq 0$$
and writing
\begin{align*}
A_j^{p_{j}-1}& = 2\,\int_{\Rn}\frac{|\phi(y)-\phi(x_j)|^{p_{j}-2}\bigl(\phi(y)-\phi(x_j)\bigr)^{+}}{|y-x_j|^{\alpha p_j}}\,dy,\\
B_j^{p_{j}-1}& =
2\int_{\Rn}\frac{|\phi(y)-\phi(x_j)|^{p_{j}-2}\bigl(\phi(y)-\phi(x_j)\bigr)^{-}}{|y-x_j|^{\alpha p_j}}\,dy,\\
C_j^{p_{j}-1}& = \lambda_{p_{j}}u_j^{p_j-1}(x_j),
\end{align*}
we get the abbreviated form
\begin{equation}
\label{abc}
A_j^{p_{j}-1} + C_j^{p_{j}-1} \leq B_j^{p_{j}-1}.
\end{equation}
According to \cite[Lemma 6.5]{CLM11} and Proposition \ref{limlambda}
$$A_j \to \mathcal{L}_{\infty}^{+}\phi\,(x_0), \quad
B_j \to  -\,\mathcal{L}_{\infty}^{-}\phi\,(x_0), \quad
C_j \to \Lambda_{\infty}^{\alpha} \phi(x_0).   $$
By dropping either $ A_j^{p_{j}-1}$ or $ C_j^{p_{j}-1} $ in
(\ref{abc}) and sending $j \to \infty$, we see that
\begin{enumerate}
\item $\mathcal{L}_{\infty}^{+}\phi\,(x_0) \leq
  -\,\mathcal{L}_{\infty}^{-}\phi\,(x_0)$, which is equivalent to
  $\mathcal{L}_{\infty}\phi\,(x_0) \leq 0$,
\item $\Lambda_{\infty}^{\alpha} \phi(x_0) \leq
  -\,\mathcal{L}_{\infty}^{-}\phi\,(x_0)$,  which is equivalent to  $\Lambda_{\infty}^{\alpha} \phi(x_0) +
  \mathcal{L}_{\infty}^{-}\phi\,(x_0)\leq 0$.
\end{enumerate}
This proves that we have a viscosity supersolution.

\emph{Viscosity subsolution.} This time the test function $\phi$ is
touching $u$ strictly from above at the point $x_0$. Now we get the
reversed inequality
\begin{equation*}
A_j^{p_{j}-1} + C_j^{p_{j}-1} \geq B_j^{p_{j}-1}.
\end{equation*}
We  now know that $\phi(x_0) > 0$ by Lemma \ref{pos}, since we already have
proved that $u$ is a viscosity supersolution ($\mathcal{L}_{\infty}u
\leq 0$).
 If  $\mathcal{L}_{\infty}^{-}\phi\,(x_0) +
\Lambda_{\infty}^{\alpha} \phi(x_0) \geq  0$, then the desired
inequality
$$\max\,\{\mathcal{L}_{\infty}\phi\,(x_0),\,\mathcal{L}_{\infty}^{-}\phi\,(x_0)+\Lambda_{\infty}^{\alpha}\phi(x_0)\}
\geq 0$$
follows immediately. The possibility that
  $-\,\mathcal{L}_{\infty}^{-}\phi\,(x_0) >
\Lambda_{\infty}^{\alpha} \phi(x_0) > 0$ remains. Then $B_j > 0$ for large
indices. We divide by  $B_j$ to obtain
$$\frac{ C_j^{p_{j}-1}}{B_j^{p_{j}-1}} + \frac{
  A_j^{p_{j}-1}}{B_j^{p_{j}-1}} \geq 1$$
and it follows that
$$
\frac{\mathcal{L}_{\infty}^{+}\phi\,(x_0)}{-\,\mathcal{L}_{\infty}^{-}\phi\,(x_0)}
\geq 1,\quad
\mathcal{L}_{\infty}^{+}\phi\,(x_0)\geq -\,\mathcal{L}_{\infty}^{-}\phi\,(x_0).$$
Thus $\mathcal{L}_{\infty}\phi\,(x_0)\geq 0.$ Again the desired
inequality holds. This proves that we have a viscosity subsolution.
\end{proof}

\section{Pointwise Behaviour}
Recall that the $\infty$-eigenvalue equation  was formulated for test functions. As we will see, a part of it, namely
\begin{equation*}
\Lmi u\,(x)+\ll u\,(x)\leq 0
\end{equation*}
holds pointwise in $\Omega$. This simplifies the investigations.

We need the auxiliary function $|x-x_0|^\alpha$, which acts as a fundamental solution. However it has to be truncated.
\begin{lemma}
\label{kon}
Let $\alpha <1$. The truncated ``$\alpha$-cone function'' 
$$C_{x_0,R}(x) =   \min\{ |x-x_0|^{\alpha},\,R^{\alpha}\}$$
satisfies the strict inequality $\mathcal{L}_{\infty}{C}_{x_0,R} \,(x)< 0$ at every
point \\$x\in B_R(x_0)\setminus\{x_0\}$.
\end{lemma}
\begin{proof}
The following estimate holds for $\Lmi$:
$$
\Lmi C_{x_0,R}\,(x)\leq \frac{C_{x_0,R}(x_0)-C_{x_0,R}(x)}{|x_0-x|^\alpha}=-1.
$$ 
In order to estimate $\Lpi$ we first remark that, since $\alpha<1$,
$$
\frac{C_{x_0,R}(y)-C_{x_0,R}(x)}{|y-x|^\alpha}\to 0, \quad \textup{as  $y\to x$}.
$$
For $x\neq y$
$$
\frac{C_{x_0,R}(y)-C_{x_0,R}(x)}{|y-x|^\alpha}\leq \frac{|y-x_0|^\alpha-|x-x_0|^\alpha}{|y-x|^\alpha}<1,
$$
where we have used the inequality
$$
C_{x_0,R}(y)\leq |y-x_0|^\alpha=|x-x_0+y-x|^\alpha<|x-x_0|^\alpha+|y-x|^\alpha,
$$
which is strict when $\alpha\in (0,1)$, $x\neq x_0$ and $y\neq x$. Hence
$$
\Lpi C_{x_0,R}\,(x)<1,
$$
and the result follows.
\end{proof}
When $\alpha=1$ the cone needs to be adjusted in order to become a strict supersolution.
\begin{lemma}\label{kon1} Let $\alpha=1$. The truncated Lipschitz cone
$$
C_{x_0,R}(x)= \min\{ |x-x_0|-\e|x-x_0|^2,\,R-\e R^2\},
$$
with $\e R<1$ satisfies $\Lmi C_{x_0,R}\, (x)<0$ at every point $x\in B_R(x_0)\setminus\{x_0\}$.
\end{lemma}
\begin{proof} The computation is the same as when $\alpha <1$.
\end{proof}

\begin{lemma}\label{Lminus} If $u\in C_0(\R^n)$ is a viscosity supersolution of the equation $\Li u=0$ in an open set $D$ where $u>0$, and if $u\leq 0$ in $\R^n\setminus D$, then 
$$
\Lmi u\,(x)=\inf_{y\in \R^n\setminus D}\frac{u(y)-u(x)}{|x-y|^\alpha}.
$$
In other words, the infimum is attained in the complement of $D$ and thus $\Lmi u$ is continuous in $D$.
\end{lemma}
\begin{remark} In general, $\Lpi u$ is not continuous.
\end{remark}
\begin{proof}
Take $x\in D$ and define
$$
L_x^-=\inf_{y\in \R^n\setminus D}\frac{u(y)-u(x)}{|y-x|^\alpha}.
$$
By the hypothesis, $L_x^-<0$. Let
$$
w(y)=u(x)+L_x^-C_{x,R}(y), 
$$
where $C_{x,R}$ is as in Lemma \ref{kon} or Lemma \ref{kon1} with $R$ chosen so that $D\subset B_R(x)$. We now claim that $u\geq w$ in $D$, which implies the lemma. In order to use the comparison principle in the open set $D\setminus \{x\}$ we see that
\begin{enumerate}
\item $\Li w>0$ in $D\setminus \{x\}$ from Lemma \ref{kon},
\item $\Li u\leq 0$ in $D$,
\item $u\geq w$ in $\R^n\setminus D$,
\item $u(x)=w(x)$.
\end{enumerate}
By the comparison principle in \cite{CLM11}, $u\geq w$. Indeed, if there
is $x_0\in D\setminus \{x\}$ such that $u(x_0)<w(x_0)$ then, for a
suitable constant $C$, $w-C$ touches $u$ from below at $x_0$,
contradicting (1) above. (Remark \ref{mono} is valid also for $p = \infty$.)
\end{proof}
As a consequence any viscosity supersolution is locally $\alpha$-H\"older continuous.
\begin{cor}\label{alphaholder} Under the hypotheses in Lemma
  \ref{Lminus} $u$ is locally $\alpha$-H\"older continuous in $D$. So
  is, in particular, a first $\infty$-eigenfunction.
\end{cor}

\begin{prop}\label{Lpw} Suppose that $u\in C_0(\R^n)$ is a viscosity solution of the $\infty$-eigenvalue equation \eqref{limeq} in an open set $D$. In addition, assume $u> 0$ in $D$ and $u\leq 0$ in $\R^n\setminus D$. If $\Lmi u\,(x_0)+\ll u(x_0)<0$ at some $x_0\in D$ in the pointwise sense, then $\Li u\,(x_0)=0$ in the viscosity sense.
\end{prop}
\begin{proof} Since we already know that $\Li u\leq 0$ in the viscosity sense, it remains only to prove that $\Li u\,(x_0)\geq 0$. Assume
$$
-\e_0=\Lmi u\,(x_0)+\ll u(x_0)<0
$$
and pick $y_0\in \R^n\setminus D$ such that 
$$
\Lmi u\,(x_0)=\frac{u(y_0)-u(x_0)}{|y_0-x_0|^\alpha}.
$$
This is possible due to Lemma \ref{Lminus}.

Let $\varphi\in C_0^1(\R^n)$ be a function touching $ u$ from above at $x_0$ and choose $\varphi_0\in C_0^1(\R^n)$ so that $\varphi\geq \varphi_0\geq u$ and
$$
\frac{\varphi_0(y_0)-u(y_0)}{|y_0-x_0|^\alpha}<\e_0.
$$
Then
\begin{align}
\Lmi  \varphi_0\,(x_0)+\ll \varphi_0(x_0)&\leq \frac{\varphi_0(y_0)-\varphi_0(x_0)}{|y_0-x_0|^\alpha}+\ll \varphi_0(x_0)\nonumber \\
&=\frac{u(y_0)-u(x_0)+\varphi_0(y_0)-u(y_0)}{|y_0-x_0|^\alpha}+\ll u(x_0)\label{Lphi0}\\
&< -\e_0+\e_0=0.\nonumber
\end{align}
But on the other hand, $\varphi_0$ touches $u$ from above at $x_0$. Hence, \eqref{Lphi0} implies $\Li \varphi_0\, (x_0)\geq 0$. Since $\varphi$ touches $\varphi_0$ from above at $x_0$, the monotonicity of $\Li$ (cf. Remark \ref{mono}) implies $\Li \varphi\, (x_0)\geq 0$.
\end{proof}

\begin{prop} Let $\alpha<1$. \label{Lmipw} Suppose that $u\in C_0(\R^n)$ is a viscosity supersolution of \eqref{limeq} in $D$, $u>0$ in $D$, and $u\leq 0$ in $\R^n\setminus D$. If there is  a $\varphi\in C_0^1(\R^n)$ touching $u$ from below at $x_0\in D$, then $\Lmi u\, (x_0)+\ll u(x_0)\leq 0$ in the pointwise sense.
\end{prop}
\begin{proof} We would like to take $u$ itself as a test function, but
  this is not allowed. Instead we construct a test function looking
  like an $\alpha$-cone with (negative) opening $\Lmi u\,(x_0)$. The details are spelled out below.

Since $\varphi$ is $C^1$ we can choose $\delta$ so small that
$$
\varphi(x)-\varphi(x_0)>\Lmi u\,(x_0)|x-x_0|^\alpha \quad \textup{ in $B_{2\delta}(x_0)$}.
$$
Choose $R$ very large and let $\psi_\delta$ be  a regularised version of 
$$\Lmi u\,(x_0) C_{x_0,R}+\varphi(x_0)$$
such that 
$$\psi_\delta =\Lmi u\,(x_0) C_{x_0,R}+\varphi (x_0) \text{ in $\R^n\setminus B_\delta(x_0)$},\quad \psi_\delta\leq \Lmi u\,(x_0)C_{x_0,R}+\varphi(x_0),$$
where $C_{x_0,R}$ is the truncated $\alpha$-cone in Lemma \ref{kon}. By definition
$$
\psi_\delta\leq\Lmi u\,(x_0) C_{x_0,R}+u(x_0)\leq u.
$$
Let $\eta_\delta$ be a cut-off function:
$$
\eta_\delta\geq 0,\quad \eta_\delta =0\textup{ in $\R^n\setminus B_{2\delta}(x_0)$}, \quad \eta_\delta = 1\textup{ in $B_\delta(x_0)$}.
$$
Finally, define $\Psi=\eta_\delta\varphi+(1-\eta_\delta)\psi_\delta$. One can verify that
$$
u\geq \Psi \geq \Lmi u\,(x_0)C_{x_0,R}+u(x_0),\quad u(x_0)=\Psi(x_0)=\varphi(x_0).
$$
In other words, $\Psi$ touches $u$ from below at $x_0$, and we can conclude
\begin{align*}
0\geq \Lmi \Psi\,(x_0)+\ll \Psi(x_0)&=\ll u(x_0)+\inf_{y\in \R^n} \frac{\Psi(y)-u(x_0)}{|y-x_0|^\alpha}\\
&\geq \ll u(x_0)+\inf_{y\in \R^n} \frac{\Lmi u\,(x_0)C_{x_0,R}(y)}{|y-x_0|^\alpha}\\
&=\ll u(x_0)+\Lmi u\,(x_0),
\end{align*}
since $\Lmi u\,(x_0) < 0$.
\end{proof}
Since a continuous function can be touched from below in a dense subset this implies, in view of Lemma \ref{Lminus} that the inequality is true everywhere.
\begin{cor} Let $\alpha<1$. Suppose $u\in C_0(\R^n)$ is a non-negative viscosity solution of the $\infty$-eigenvalue equation \eqref{limeq} in $D$, $u> 0$  in $D$, and $u\leq 0$ in $\R^n\setminus D$. Then $\Lmi u+\ll u\leq 0$ in $D$ in the pointwise sense.
\end{cor}
\begin{proof} Part d) of Lemma 1.8 in \cite{BCD97} states that the subdifferential of a continuous function is non-empty in a dense subset. That the subdifferential is non-empty is equivalent to the existence of a $C^1$-function touching from below. Thus, from Lemma \ref{Lmipw}, $\Lmi u+\ll u\leq 0$ holds in a dense subset of $D$. By Lemma \ref{Lminus}, $\Lmi u+\ll u$ is a continuous function and hence the inequality holds in the whole $D$.
\end{proof}
When $\alpha=1$ the proof has to be modified slightly.
\begin{prop} Let $\alpha=1$. If $u\in C_0(\R^n)$ is a non-negative viscosity solution of \eqref{limeq}, $u> 0$  in $D$, and $u\leq 0$ in $\R^n\setminus D$. Then $\Lmi u+\Lambda_\infty u\leq 0$ in $D$ in the pointwise sense.
\end{prop}
\begin{proof} By Corollary \ref{alphaholder}, $u$ is locally Lipschitz continuous and thus by Rade-macher's theorem, $u$ is a.e. differentiable. Take $x_0$ where $u$ is differentiable. Then it is well known that one can find a $C^1$ function $\varphi$ touching $u$ from below at $x_0$. Moreover, $\Lmi u\,(x_0)\leq -|\nabla u(x_0)|=-|\nabla \varphi(x_0)|$. Given $\e>0$, there is $\delta>0$ such that
$$
\varphi	(x)\geq \varphi(x_0)+[\Lmi u\,(x_0)-\e]\,C_{x_0,R}(x) \textup{ in $B_{2\delta}(x_0)$}.
$$
Repeating the procedure with  $\eta_\delta,\psi_\delta$ and $\Psi$ as in the proof of Proposition \ref{Lmipw}, we obtain that 
$$
\Lmi u\,(x_0)+\Lambda_\infty u(x_0)\leq \e.
$$
Since $\e$ was arbitrary, this yields $\Lmi u\,(x_0)+\Lambda_\infty u(x_0)\leq 0$, and this holds at a.e. point in $D$. By Lemma \ref{Lminus}, $\Lmi u+\Lambda_\infty u$ is  a continuous function, so this must hold everywhere.
\end{proof}

\section{The Ground State} Recall that the  first
$\infty$-eigenfunctions were defined in Definition \ref{infeigen} as the non-negative
solutions in  $C_0(\overline{\Omega})$ of  the $\infty$-eigenvalue equation
\eqref{limeq}. We will give a remarkable representation formula for
one first $\infty$-eigenfunction, valid in any domain. In some cases we can
assure uniqueness.

We need some concepts related to the geometry of $\Omega$. We denote
by $\delta(x)$ the distance function, $\dist(x,\Rn \setminus
\Omega)$. This function is Lipschitz continuous and $|\nabla \delta| =
1 $ almost everywhere in $\Omega$. We define the \emph{High Ridge} as the set of points where the distance function attains its maximum, i.e.
$$
\hr=\{x\in \Omega |\,\delta(x)=R\},
$$
where as before, $R$ denotes the radius of the largest ball that can
be inscribed inside $\Omega$. The function $\delta(x)$ is not
differentiable on $\Gamma$. The High Ridge is a closed set and $\Omega
\setminus \Gamma$ is open. We denote 
$$\rho(x) = \dist(x,\Gamma).$$

The quantity $\Lambda_{\infty}^{\alpha}$ behaves as a genuine
eigenvalue in the sense that it cannot be replaced by any other number
in the $\infty$-eigenvalue equation:

\begin{thm} Let $u\in C_0(\overline\Omega),\,u\not \equiv 0$, be a non-negative solution of 
$$
\max\left\{\Li u\,(x),\Lmi u\, (x)+\lambda u(x)\right\}=0\quad \textup{in $\Omega$.}
$$
Then  $\lambda=\ll$.
\end{thm}
\begin{proof} From Proposition \ref{Lmipw}, $\Lmi u+\lambda u\leq 0$ in the pointwise sense and from Lemma \ref{Lminus}
$$
\Lmi u\,(x)=\inf_{y\in \R^n\setminus\Omega}\frac{u(y)-u(x)}{|y-x|^\alpha}=-\frac{u(x)}{\delta(x)^\alpha}.
$$
Eliminating $u$ from the inequality, we obtain that $\lambda \leq \frac{1}{\delta(x)^\alpha}$ for all $x\in \Omega$. Hence, $\lambda\leq\ll$.

Now, assume  that $\lambda < \ll$. Then
$\lambda<\frac{1}{\delta(x)^\alpha}$ and thus $\Lmi u\,(x)+\lambda
u(x)<0$ for all $x\in \Omega$. By Lemma \ref{Lpw}, $\Li u =0$ in
$\Omega$, which by the comparison principle (\cite[Prop.~11.2]{CLM11})
implies that $u$ is identically zero in $\Omega$. This case was
excluded. Hence $\lambda \geq \Lambda_{\infty}^{\alpha}.$ The result follows.
\end{proof}
An immediate consequence of the theorem above is that any first $\infty$-eigenfunction minimizes the Rayleigh quotient \eqref{minimum}.

The fundamental role of the High Ridge $\Gamma$ is revealed in:

\begin{thm}\label{hrlem} Let $u$ be a first
  $\infty$-eigenfunction. Then $\Lmi u\,(x)+\ll u\,(x)=0$ (pointwise)
  if and only if $x\in \hr$. In the complement $\Omega \setminus
  \Gamma$ the equation $\Li u=0$  holds in the viscosity sense.
\end{thm}
\begin{proof} By Lemma \ref{Lminus}
$$
\Lmi u\,(x)=\inf_{y\in \R^n\setminus\Omega}\frac{u(y)-u(x)}{|y-x|^\alpha}=\inf_{y\in \R^n\setminus\Omega}\frac{-u(x)}{|y-x|^\alpha}=-\frac{u(x)}{\delta(x)^\alpha}.
$$
Thus
$$
\Lmi u\,(x)+\ll
u\,(x)=u(x)\left(\frac{1}{R^\alpha}-\frac{1}{\delta(x)^\alpha}\right)
\leq 0
$$
with equality if and only if $\delta(x)=R$, i.e., if and only if $x\in \hr$.
\end{proof}

This provides us with a method to construct   first
$\infty$-eigenfunctions, using an equation that does not explicitly
contain $\Lambda_{\infty}^{\alpha}$. Let $\Gamma_{1} \subset \Gamma$
be an arbitrary closed non-empty subset. According to \cite[Thm~1.5]{CLM11}, the Dirichlet boundary value problem 
\begin{equation}\label{explicitfcn}
\left\{\begin{array}{lr} \Li u =0 & \textup{in $\Omega \setminus \Gamma_{1}$},\\
u=0 & \textup{in $\R^n\setminus \Omega$},\\
u=1 & \textup{on $\Gamma_{1}$},
\end{array}\right.
\end{equation}
has a unique viscosity solution in $\Omega\setminus \Gamma_{1}$, which
takes the boundary values continuously. Moreover, $0\leq u\leq 1$ by
\cite[Prop.~11.2]{CLM11}. Moreover, by Lemma \ref{positivity} we have $0 < u < 1$ in $\Omega
\setminus \Gamma_{1}$. 
Therefore different subsets yield different solutions!
\begin{thm}\label{explicit} The solution of the Dirichlet problem \eqref{explicitfcn} is a first $\infty$-eigenfunction in $\Omega$.
\end{thm}
\begin{proof} We first prove that $u$ is a viscosity supersolution of
  the $\infty$-eigenvalue equation \eqref{limeq}. Take $\varphi$ touching $u$ from below at $x_0\in \hr_1$. Then by direct pointwise computations
$$
\Li \varphi\,(x_0)\leq \Li u\,(x_0)\leq 0, 
$$
since $0\leq u\leq 1$.

Hence $u$ is a supersolution of $\Li u\leq 0$  in the whole $\Omega$. By Lemma \ref{Lminus}
$$
\Lmi u\,(x)=\inf_{y\in \R^n\setminus \Omega}\frac{u(y)-u(x)}{|y-x|^\alpha}=-\frac{u(x)}{\delta(x)^\alpha}
$$
and we can conclude
$$
\Lmi u\,(x)+\ll u(x)=u(x)\left(\frac{1}{R^{\alpha}}-\frac{1}{\delta(x)^\alpha}\right)\leq 0,
$$
for any $x\in \Omega$, since $R \geq \delta(x)$. Thus $u$ is a
viscosity supersolution of the $\infty$-eigenvalue equation \eqref{limeq} in $\Omega$.

To prove that $u$ is also a viscosity subsolution of \eqref{limeq}, it is enough to verify that $\Lmi u+\ll u\geq 0$ on $\hr_1$. This follows by the same arguments as in the proof of Lemma \ref{hrlem}.
\end{proof}
 We can give the solution of \eqref{explicitfcn} explicitly in terms
 of  distances.
\begin{thm}[Representation Formula] \label{expliciteigfcn} Let $\rho_{1}(x)=\dist(x,\Gamma_{1})$. The function
$$
u(x)=\frac{\delta(x)^\alpha}{\delta(x)^\alpha+\rho_{1}(x)^\alpha}
$$
solves the problem \eqref{explicitfcn} and is therefore a first $\infty$-eigenfunction.
\end{thm}
\begin{proof}
For notational convenience, we drop the index  writing $\Gamma$ for
$\Gamma_{1}$ and $\rho$ for $\rho_{1}$ in this proof.
We first claim that when $x\in \Omega\setminus \Gamma$, the supremum
in $\Lpi u\,(x)$ is attained on $\Gamma$ or, in other words, that
\begin{align*}
\frac{u(y)-u(x)}{|y-x|^\alpha}&\leq \sup_{y\in \Gamma}\frac{u(y)-u(x)}{|y-x|^\alpha}\\
&=\sup_{y\in \Gamma}\frac{1-\frac{\delta(x)^\alpha}{\delta(x)^\alpha+\rho(x)^\alpha}}{|y-x|^\alpha}\\&=\displaystyle\frac{\frac{\rho(x)^\alpha}{\delta(x)^\alpha+\rho(x)^\alpha}}{{\underbrace{|y_x-x|}_{=\rho(x)}}^\alpha}\\
&=\frac{1}{\delta(x)^\alpha+\rho(x)^\alpha}
\end{align*}
for all $y\in \Omega$. This is equivalent to
$$
\frac{\delta(y)^\alpha\rho(x)^\alpha-\delta(x)^\alpha\rho(y)^\alpha}{|x-y|^\alpha(\delta(x)^\alpha+\rho(x)^\alpha)}\leq 1, 
$$
or
$$
\delta(y)^\alpha\rho(x)^\alpha\leq |x-y|^\alpha\delta(x)^\alpha+|x-y|^\alpha\rho(x)^\alpha+\delta(x)^\alpha\rho(y)^\alpha.
$$
Since $\alpha\in (0,1]$,  the triangle inequality yields
$$
\delta(y)^\alpha\leq |x-y|^\alpha+\delta(x)^\alpha,\quad \rho(x)^\alpha\leq \rho(y)^\alpha+|x-y|^\alpha.
$$
Hence, 
$$
\delta(y)^\alpha\rho(x)^\alpha\leq |x-y|^\alpha\rho(x)^\alpha+\delta(x)^\alpha\rho(x)^\alpha\leq |x-y|^\alpha\rho(x)^\alpha+\delta(x)^\alpha(\rho(y)^\alpha+|x-y|^\alpha),
$$
which proves the claim.

It remains to verify that $u$ solves \eqref{explicitfcn}. Take $x\in \Omega\setminus \hr$. From the claim
$$
\Lpi u\,(x)=\frac{1}{\delta(x)^\alpha+\rho(x)^\alpha}.
$$
Moreover, 
$$
\Lmi u\,(x)\leq  \inf_{y\in \R^n\setminus
  \Omega}\frac{u(y)-u(x)}{|y-x|^\alpha} = \frac{-u(x)}{{\underbrace{|y_x-x|}_{=\delta(x)}}^\alpha}=-\frac{1}{\delta(x)^\alpha+\rho(x)^\alpha}.
$$
Thus  $\Li u \leq 0$ in $\Omega \setminus \Gamma.$ 
  If $x\in \Gamma$
then $\Lmi u\,(x)\leq \Lpi u\,(x)\leq 0$ since $u$ attains its maximum
there. Thus, $\Li u\leq 0$ in $\Omega$. By Lemma \ref{Lminus} the
infimum in $\Lmi u\,(x)$ is attained in $\R^n\setminus \Omega$ so that 
$$
\Lmi u\,(x)=-\frac{1}{\delta(x)^\alpha+\rho(x)^\alpha}.
$$
Hence, $\Li u =0$ in $\Omega\setminus\hr$. The boundary values of $u$
on $\Gamma$ and $\Rn \setminus \Omega$ are $1$ and $0$. Thus $u$ is a
solution of the Dirichlet problem \eqref{explicitfcn}. The final result
follows from Theorem \ref{explicit}.
\end{proof}

In the case when every first $\infty$-eigenfunction is constant on the
High Ridge,  the first $\infty$-eigenfunction is unique (up to
multiplication by a constant). It is the solution given in Theorem
\ref{explicit}. Indeed, we have:
\begin{cor} \label{constonridge} A first $\infty$-eigenfunction that is constant on $\hr$ is given by the representation formula 
$$
u(x)=C\frac{\delta(x)^\alpha}{\delta(x)^\alpha+\rho(x)^\alpha}.
$$
\end{cor}
\begin{proof} Let $u$ be a first $\infty$-eigenfunction. By Lemma
  \ref{hrlem}, $\Li u =0$ outside $\hr$ so that, up to a multiplicative
  constant, $u$ satisfies \eqref{explicitfcn}. By \cite[Thm~1.5]{CLM11}
  the solution of equation  \eqref{explicitfcn} is unique.
\end{proof}

\medskip

\emph{Example:} This certainly implies uniqueness when the High Ridge
consists of only one point, as for a ball or a cube. The first
eigenfunction for the ball $B(0,R)$ is
$$ \frac{(R-|x|)^\alpha}{(R-|x|)^\alpha+|x|^\alpha}.$$
For $\alpha=1$ it becomes $\delta(x)=R-|x|$, which incidentally also
solves the differential equation \eqref{globq}.

\medskip

 We have seen that if the High Ridge $\Gamma$ consists
  of more than one point, we can construct several linearly independent first
  $\infty$-eigenfunctions for the same domain $\Omega$. It stands to
  reason that the limiting procedure $u_{p_{j}} \to u$ in Section \ref{sec:passage}
yields the maximal solution, in which case $\Gamma_{1} = \Gamma$. We
have no valid proof, except in some symmetric special cases.

We have a geometric criterion to guarantee that the distance function is a first $\infty$-eigenfunction.
\begin{cor} Take $\alpha =1$. If the distance function is differentiable outside $\hr$, then the distance function is a first $\infty$-eigenfunction.
\end{cor}
\begin{proof}
The first step is to control $\Li \delta$. Since $\delta$ is differentiable outside $\hr$, $|\nabla \delta |=1$ there. Moreover, $\delta$ is Lipschitz continuous with constant 1. Therefore with $h=\nabla \delta(x)$ for $x\not\in \hr$
$$
1\geq \sup_{y\in \R^n} \frac{\delta(y)-\delta(x)}{|y-x|}\geq \lim_{t\searrow 0}\frac{\delta(x+th)-\delta(x)}{|h|}=1
$$
and
$$
-1\leq \inf_{y\in \R^n} \frac{\delta(y)-\delta(x)}{|y-x|}\leq \lim_{t\nearrow 0}\frac{\delta(x+th)-\delta(x)}{|h|}=-1.
$$
Thus, $\Lpi \delta\,(x)=-\Lmi \delta\,(x)=1$ or equivalently $\Li \delta \,(x)=0$ for $x\not\in \hr$. The result now follows from Theorem \ref{explicit}.

\end{proof}

\section{Higher Infinity Eigenvalues} Also for the higher eigenfunctions it is possible to deduce a limiting equation as $p\to\infty$. The equation is the one for the first eigenfunction in every nodal domain together with a transition condition:
\begin{equation}\label{limeqhigh}
\left\{\begin{array}{lr} \max\left\{\mathcal{L}_{\infty}u\,(x),\,\mathcal{L}_{\infty}^{-}u\,(x)
    + \lambda u(x)\right\}\,=\,0& \textup{when $u(x)>0$}\\
    \Li u\,(x)\,=\,0& \textup{when $u(x)=0$}\\
\min\left\{\mathcal{L}_{\infty}u\,(x),\,\mathcal{L}_{\infty}^{+}u\,(x)
    + \lambda u(x)\right\}\,=\,0& \textup{when $u(x)<0$}
    \end{array}\right.
\end{equation}
The result below can be obtained by following the proof of Theorem \ref{convergence}.
\begin{thm} Let $u_p$ be a sign-changing eigenfunction with the finite
  exponent $p$. Then, upon normalizing $u_p$, there is a subsequence $u_{p_j}$ converging uniformly in $\Omega$ to a function $u\in C_0(\overline\Omega)$ which is a viscosity solution of equation \eqref{limeqhigh} for some $\lambda\geq \Lambda_{\infty}^{\alpha}(\Omega)$.
\end{thm}
This leads to the following definition of higher $\infty$-eigenfunctions:
\begin{definition} We say that $u\in C_0(\overline\Omega)$ is a higher $\infty$-eigenfunction with eigenvalue $\lambda $ if $u$ is a sign-changing viscosity solution of equation \eqref{limeqhigh}.
\end{definition}
We give a list of properties that hold for higher $\infty$-eigenfunctions, which can be proved in the same manner as those for the first $\infty$-eigenfunctions:
\begin{itemize}
\item The infimum in $\mathcal{L}_{\infty}^- u$ is attained in the set $\{ u\leq 0\}$ and the supremum in $\mathcal{L}_{\infty}^+ u$ is attained in the set $\{ u\geq 0\}$. Follows from Lemma \ref{Lminus}.
\item $\Li u =0$ in the viscosity sense wherever
$$\mathcal{L}_{\infty}^-u+\lambda u<0\textup{ and } u>0$$
or
$$\mathcal{L}_{\infty}^+u+\lambda u>0\textup{ and } u<0.$$
See Proposition \ref{Lpw}. 
\item When $u>0$ then $  \mathcal{L}_{\infty}^-+\lambda u\leq 0$ in the pointwise sense, and when $u<0$, then $  \mathcal{L}_{\infty}^++\lambda u\geq 0$, also in the pointwise sense. See Proposition \ref{Lmipw}.
\end{itemize}

We change the notation now so that $R_1$ denotes the radius of the
largest inscribed ball in $\Omega$. We define $R_2=R_2(\Omega)$ as the
largest radius $R$ such that two disjoint open balls of radius $R$ can be inscribed in $\Omega$.
\begin{prop} If $u$ is a higher $\infty$-eigenfunction with eigenvalue $\lambda$ then
$$
\lambda\geq \frac{1}{R^\alpha_2}.
$$
\end{prop}
\begin{proof} Pick $x_0\in \{u>0\}$ such that
$$
\lambda u(x_0)+\Lmi u\,(x_0)=0.
$$
Such an $x_0$ exists since otherwise, by Proposition \ref{Lpw}, $\Li u=0$ in $\{u>0\}$, which by the comparison principle in \cite{CLM11} would force $u\leq 0$.

Since $\Lmi u\,(x_0)$ is attained in $ \{u\leq 0\}$ (cf. Property 1 above)
\begin{align*}
\lambda u(x_0) &= -\Lmi u\,(x_0)\\
&=-\inf_{y\in \R^n\cap \{u\leq 0\}}\frac{u(y)-u(x_0)}{|y-x_0|^\alpha}\\
&\geq -\inf_{y\in \R^n\cap \dd\{u> 0\}}\frac{u(y)-u(x_0)}{|y-x_0|^\alpha}\\
&\geq \frac{u(x_0)}{\dist(x_0,\dd\{u\leq 0\})^\alpha}.
\end{align*}
The same can be obtained for $y_0\in \{u<0\}$ so that
$$
\lambda \geq \max\left(\sup_{x_0\in \{u>0\}}\frac{1}{\dist(x_0,\dd\{u> 0\})^\alpha},\sup_{y_0\in \{u<0\}}\frac{1}{\dist(y_0,\dd\{u< 0\})^\alpha}\right)\geq \frac{1}{R_2^\alpha}.
$$
\end{proof}

The above proposition implies that in the case when $R_2\neq R_1$ we can define the second eigenvalue as
$$
\inf\{\lambda: \lambda \textup{ is an eigenvalue of $u$}, \, u \textup{ changes signs}\}.
$$
There are simple examples of domains with $R_1 = R_2$.
If $\alpha<1$ and if there is a nodal domain compactly contained in
$\Omega$, we are able to obtain a better lower bound for the second
eigenvalue.
We encounter a strange phenomenon when $\alpha \neq 1$, viz. the
restriction of a higher $\infty$-eigenfunction to a nodal domain (and
extended as zero) is not a first $\infty$-eigenfunction with respect
to the nodal domain.
\begin{prop} Assume $u$ to be a higher $\infty$-eigenfunction with eigenvalue $\lambda$. If $N$ is a nodal domain compactly contained in the interior of $\Omega$, then $\lambda>\ll(N)$.
\end{prop}
\begin{proof} We can assume that $\{u>0\}$ in $N$. As before we can find $x_0\in N$ such that $\Lmi u\,(x_0)+\lambda u(x_0)=0$. Since $\Lmi u\,(x_0)$ is attained in $\{u\leq 0\}$,
\begin{align}\label{lambdaell}
\lambda \geq -\frac{\Lmi u\,(x_0)}{u(x_0)}\geq \inf_{y\in \dd N}\frac{1}{|y-x_0|^\alpha}=\frac{1}{\dist(x_0,\dd N)^\alpha}\geq \ll(N).
\end{align}
Assume now towards a contradiction that $\lambda = \ll(N)$. Then equality holds all the way in \eqref{lambdaell}, so that
$$
\Lmi u\,(x_0)=\frac{-u(x_0)}{\dist(x_0,\dd N)^\alpha},
$$
or, in other words, for all $y\in \R^n$
$$
u(y)\geq u(x_0)\left(1-\frac{|y-x_0|^\alpha}{\dist(x_0,\dd N)^\alpha}\right)
$$
with equality if $y=x_0$ or if $y=y_0$ with $|y_0-x_0|=\dist(x_0,\dd N)$. Consequently, with $R$ large enough the function
$$
w(y)=u(x_0)-\frac{u(x_0)}{\dist(x_0,\dd N)^\alpha}C_{x_0,R}(y)
$$
touches $u$ from below at $y_0$. From equation \eqref{limeqhigh}, $\Li w\,(y_0)\leq 0$. But on the other hand, Lemma \ref{kon} implies $\Li w>0$, a contradiction.
\end{proof}
\section{One Dimensional Examples}\label{sec:1d}
Certain aspects of this non-local problem differ from the situation in the eigenvalue problem \eqref{globq} for the infinity Laplacian. In the case $\alpha<1$, these differences appear explicitly in one-dimensional examples.
\subsection{The first eigenfunction}
Consider the interval $(0,2)$. Its High Ridge consists only of the
midpoint, and by Corollary \ref{constonridge} the first eigenfunction
is unique and given by the representation formula in Lemma \ref{expliciteigfcn}. In the case of the interval $(0,2)$ it reduces to
$$
u(x)=\frac{\min(|x|^\alpha,|2-x|^\alpha)}{\min(|x|^\alpha,|2-x|^\alpha)+|x-1|^\alpha}.
$$
\begin{figure}[!ht]
\begin{center}
     \includegraphics[scale=0.3]{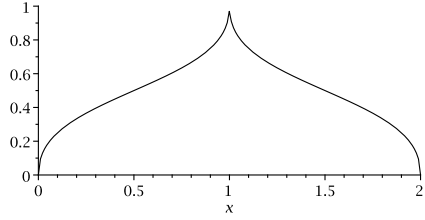}

    \caption{The first eigenfunction on $(0,2)$, for $\alpha=1/2$.}
\end{center}
\end{figure}
\subsection{The second eigenfunction}
Consider the interval $(0,2)$. Assuming that the function $u$ is anti-symmetric around the point $x=1$ one can construct a solution having two nodal domains:
\begin{figure}[!ht]
\begin{center}
     \includegraphics[scale=0.3]{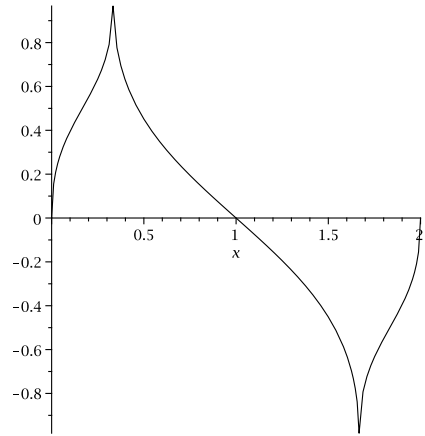}

    \caption{One eigenfunction on $(0,2)$ with two nodal domains for $\alpha=1/2$ and $\lambda=\sqrt{3}$.}
\end{center}
\end{figure}
$$
u(x)=\displaystyle\left\{ \begin{array}{ll} 
\displaystyle\frac{x^\alpha}{x^\alpha+(a-x)^\alpha}& \textup{for $x\in (0,a)$,}\\

\displaystyle\frac{(2-a-x)^\alpha-(x-a)^\alpha}{(2-a-x)^\alpha+(x-a)^\alpha}&\textup{for $x\in (a,2-a)$,}\\

\displaystyle-\frac{(2-x)^\alpha}{(2-x)^\alpha+(x-(2-a))^\alpha}& \textup{for $x\in (2-a,2)$,}
\end{array}\right.
$$
here 
$$
a=\frac{2}{2^\frac{1}{\alpha}+2},\quad \lambda=(2^{\frac{1}{\alpha}-1}+1)^\alpha
$$
and the nodal domains are the two intervals $(0,1)$ and $(1,2)$. The
maximum is at $x=a$ and the minimum at $x=2-a$.  For
$\alpha\neq 1$, one can see that $a<1/2$. The remarkable feature is that the maximum is
not attained at the midpoint of the nodal interval $(0,1)$ but to the
left. In this example $\lambda>\ll(\{u>0\})=\ll((0,1))=1$.

\subsection{A function with three nodal domains}
Consider the interval $(0,2)$. 
Assuming that the solution is symmetric around the point $x=1$, we obtain one eigenfunction with three nodal intervals:
\begin{figure}[!ht]
\begin{center}
     \includegraphics[scale=0.3]{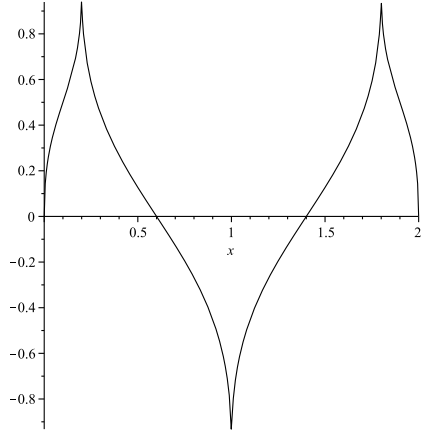}
    \caption{One eigenfunction on $(0,2)$ with three nodal domains for $\alpha=1/2$ and $\lambda=\sqrt{5}$.}
\end{center}  
\end{figure}
$$
u(x)=\displaystyle\left\{ \begin{array}{ll} 
\displaystyle\frac{x^\alpha}{x^\alpha+(a-x)^\alpha}& \textup{when $x\in (0,a)$,}\\

\displaystyle\frac{(1-x)^\alpha-(x-a)^\alpha}{(1-x)^\alpha+(x-a)^\alpha}&\textup{when $x\in (a,1)$,}\\
\end{array}\right.
$$
and $u(2-x) = u(x)$. Here
$$
a=\frac{1}{2^\frac{1}{\alpha}+1},\quad \lambda=(1+2^\frac{1}{\alpha})^\alpha
$$
and the nodal intervals are $(0,\frac{1+a}{2})$,
$(\frac{1+a}{2},\frac{3-a}{2})$ and $(\frac{3-a}{2},2)$. The
remarkable feature is that the nodal intervals do not have the same
length. The middle interval is the longest. This illustrates that nodal domains (coming from the same
eigenfunction) can have different first $\infty$-eigenvalues.

\bibliographystyle{amsrefs}
\bibliography{eigref}

\begin{bibdiv}
\begin{biblist}

\bib{Ana87}{article}{
      author={Anane, Aomar},
       title={Simplicit\'e et isolation de la premi\`ere valeur propre du
  {$p$}-laplacien avec poids},
        date={1987},
     journal={C. R. Acad. Sci. Paris S\'er. I Math.},
      volume={305},
      number={16},
       pages={725\ndash 728},
}

\bib{BBM02}{article}{
      author={Bourgain, Jean},
      author={Brezis, Ha{\"{\i}}m},
      author={Mironescu, Petru},
       title={Limiting embedding theorems for {$W^{s,p}$} when {$s\uparrow1$}
  and applications},
        date={2002},
     journal={J. Anal. Math.},
      volume={87},
       pages={77\ndash 101},
        note={Dedicated to the memory of Thomas H. Wolff},
}

\bib{BCD97}{book}{
      author={Bardi, Martino},
      author={Capuzzo-Dolcetta, Italo},
       title={Optimal {C}ontrol and {V}iscosity {S}olutions of
  {H}amilton-{J}acobi-{B}ellman {E}quations},
      series={Systems \& Control: Foundations \& Applications},
   publisher={Birkh\"auser Boston Inc.},
     address={Boston, MA},
        date={1997},
        note={With appendices by Maurizio Falcone and Pierpaolo Soravia},
}

\bib{BK02}{article}{
      author={Belloni, Marino},
      author={Kawohl, Bernd},
       title={A direct uniqueness proof for equations involving the
  {$p$}-{L}aplace operator},
        date={2002},
     journal={Manuscripta Math.},
      volume={109},
      number={2},
       pages={229\ndash 231},
}

\bib{CPJ09}{article}{
      author={Champion, Thierry},
      author={De~Pascale, Luigi},
      author={Jimenez, Chlo{\'e}},
       title={The {$\infty$}-eigenvalue problem and a problem of optimal
  transportation},
        date={2009},
     journal={Commun. Appl. Anal.},
      volume={13},
      number={4},
       pages={547\ndash 565},
}

\bib{CLM11}{article}{
      author={Chambolle, Antonin},
      author={Lindgren, Erik},
      author={Monneau, R\'egis},
       title={A {H}\"older infinity {L}aplacian},
        date={2011},
     journal={accepted for publication in ESAIM: Control, Optimisation and
  Calculus of Variations},
         url={http://dx.doi.org/10.1051/cocv/2011182},
}

\bib{NPV11}{article}{
      author={Di~Nezza, Eleonora},
      author={Palatucci, Giampiero},
      author={Valdinoci, Enrico},
       title={Hitchhiker’s guide to the fractional {S}obolev spaces},
        date={2011},
     journal={preprint},
}

\bib{FL11}{article}{
      author={Frank, Rupert~L.},
      author={Leander, Giesinger},
       title={Refined semiclassical asymptotics for fractional powers of the
  {L}aplace operator},
        date={2011},
     journal={preprint},
}

\bib{IN10}{article}{
      author={Ishii, Hitoshi},
      author={Nakamura, Gou},
       title={A class of integral equations and approximation of
  {$p$}-{L}aplace equations},
        date={2010},
     journal={Calc. Var. Partial Differential Equations},
      volume={37},
      number={3-4},
       pages={485\ndash 522},
}

\bib{JLM99}{article}{
      author={Juutinen, Petri},
      author={Lindqvist, Peter},
      author={Manfredi, Juan~J.},
       title={The {$\infty$}-eigenvalue problem},
        date={1999},
     journal={Arch. Ration. Mech. Anal.},
      volume={148},
      number={2},
       pages={89\ndash 105},
}

\bib{Kas07}{article}{
      author={Kassmann, Moritz},
       title={The classical {H}arnack inequality fails for nonlocal operators},
        note={preprint No.360, Sonderforschungsbereich 611},
}

\bib{Kel67}{book}{
      author={Kellogg, Oliver~Dimon},
       title={Foundations of {P}otential {T}heory},
      series={Reprint from the first edition of 1929. Die Grundlehren der
  Mathematischen Wissenschaften, Band 31},
   publisher={Springer-Verlag},
     address={Berlin},
        date={1967},
}

\bib{KL06}{article}{
      author={Kawohl, Bernd},
      author={Lindqvist, Peter},
       title={Positive eigenfunctions for the {$p$}-{L}aplace operator
  revisited},
        date={2006},
     journal={Analysis (Munich)},
      volume={26},
      number={4},
       pages={545\ndash 550},
}

\bib{Koi04}{book}{
      author={Koike, Shigeaki},
       title={A {B}eginner's {G}uide to the {T}heory of {V}iscosity
  {S}olutions},
      series={MSJ Memoirs},
   publisher={Mathematical Society of Japan},
     address={Tokyo},
        date={2004},
      volume={13},
}

\bib{Kwa12}{article}{
      author={Kwa\'snicki, Mateusz},
       title={Eigenvalues of the fractional {L}aplace operator in the
  interval},
        date={2012},
     journal={Journal of Functional Analysis},
      volume={262},
      number={5},
       pages={2379 \ndash  2402},
}

\bib{OT88}{article}{
      author={{\^O}tani, Mitsuharu},
      author={Teshima, Toshiaki},
       title={On the first eigenvalue of some quasilinear elliptic equations},
        date={1988},
     journal={Proc. Japan Acad. Ser. A Math. Sci.},
      volume={64},
      number={1},
       pages={8\ndash 10},
}

\bib{Yu07}{article}{
      author={Yu, Yifeng},
       title={Some properties of the ground states of the infinity
  {L}aplacian},
        date={2007},
     journal={Indiana Univ. Math. J.},
      volume={56},
      number={2},
       pages={947\ndash 964},
}

\bib{ZRK07}{article}{
      author={Zoia, Andrea},
      author={Rosso, Alberto},
      author={Kardar, Mehran},
       title={Fractional {L}aplacian in bounded domains},
        date={2007},
     journal={Phys. Rev. E (3)},
      volume={76},
      number={2},
       pages={021116, 11},
}

\end{biblist}
\end{bibdiv}
\end{document}